\documentclass[11pt,leqno]{amsart}
\usepackage{amscd,amsfonts,mathrsfs,amsthm,amssymb}
\usepackage{mathtools}
\usepackage{comment}
\usepackage{enumerate}
\usepackage{multicol}        			
\usepackage{color}           			
\usepackage{array}
\usepackage{ae}
\usepackage{epsfig}
\usepackage[e]{esvect}
\usepackage{empheq}
\usepackage{hyperref} 
\usepackage[latin1]{inputenc}

\usepackage{calrsfs}
\DeclareMathAlphabet{\pazocal}{OMS}{zplm}{m}{n}

\newcommand{\R}{\mathbf{R}}

\newcommand{\Ff}{\mathcal F}

\newcommand{\Rr}{\mathcal R}
 
 \newcommand{\Dd}{\mathcal{D}}
 \newcommand{\Hh}{\mathcal{H}}
 
 \newcommand{\Ll}{\mathcal{L}}

  \newcommand{\Ss}{\mathcal S}
 
 \newcommand{\RR}{\mathbf{R}}  
 \newcommand{\ZZ}{\mathbf{Z}}  
 \newcommand{\BB}{\mathbf{B}}  
 \renewcommand{\SS}{\mathbf{S}}  
    
  \newcommand{\Div}{\operatorname{Div}}
  
    \newcommand{\dist}{\operatorname{dist}}
 
 \newcommand{\area}{\operatorname{area}}
 \newcommand{\eps}{\epsilon}

\renewcommand{\vv}{\mathbf v}
\newcommand{\ee}{\mathbf e}

\newcommand{\pdf}[2]{\frac{\partial #1}{\partial #2}}

\newtheorem*{theorem*}{Theorem}
\newtheorem{theorem}{Theorem}[section]

\newtheorem{lemma}[theorem]{Lemma}

\newtheorem*{claim*}{Claim}

\newtheorem{corollary}[theorem]{Corollary}
\newtheorem{proposition}[theorem]{Proposition}
\theoremstyle{definition}
\newtheorem{definition}[theorem]{Definition}
\newtheorem{claim}{Claim}
\newtheorem{remark}[theorem]{Remark}

\def\pproof#1{\@ifnextchar[\opargproof
{\opargproof[\it Proof of #1.]}}
\def\opargproof[#1]{\par\noindent {\bf #1 }}

\makeatother

\usepackage[alphabetic, msc-links, backrefs]{amsrefs}

\begin{document}

\title[Scherk-Like Translators]{Scherk-like Translators for Mean Curvature Flow}

\begin{abstract}
We prove existence and uniqueness for a two-parameter family
of translators for mean curvature flow.
We get additional examples by taking limits at
the boundary of the parameter space.   
Some of the translators resemble well-known
minimal surfaces (Scherk's doubly periodic minimal surfaces,
helicoids), but others have no minimal surface analogs.
A one-parameter subfamily of the examples (the pitchforks)
have finite topology and quadratic area growth, and thus
might arise as blowups at singularities of initially smooth, closed
surfaces flowing by mean curvature flow.
\end{abstract}

\author[D. Hoffman]{\textsc{D. Hoffman}}

\address{David Hoffman\newline
Department of Mathematics \newline
 Stanford University \newline 
  Stanford, CA 94305, USA\newline
{\sl E-mail address:} {\bf dhoffman@stanford.edu}
}

\author[F. Martin]{\textsc{F. Martín}}

\address{Francisco Martín\newline
Departmento de Geometría y Topología  \newline
Instituto de Matemáticas IE-Math Granada \newline
Universidad de Granada\newline
18071 Granada, Spain\newline
{\sl E-mail address:} {\bf fmartin@ugr.es}
}
\author[B. White]{\textsc{B. White}}

\address{Brian White\newline
Department of Mathematics \newline
 Stanford University \newline 
  Stanford, CA 94305, USA\newline
{\sl E-mail address:} {\bf bcwhite@stanford.edu}
}

\date{\today}
\thanks{The second author was partially supported by MINECO-FEDER grant no. MTM2017-89677-P and 
by the Leverhulme Trust grant IN-2016-019.
The third author was partially supported by NSF grant DMS-1711293}



\maketitle


\setcounter{tocdepth}{1}
   \tableofcontents
\section{Introduction}
\label{sec:intro}
A surface $M$ in $\R^3$ is called a {\bf translator with velocity $\vv$} 
if 
\[
  t\mapsto M + t\vv
\]
is a mean curvature flow, i.e., if the normal velocity at each point is equal
to the mean curvature vector: $\vec H = \vv^\perp$.
By rotating and scaling, it suffices to consider the case $\vv=(0,0,-1)$:
in that case, we refer to $M$ simply as a
 translator\footnote{We follow Ilmanen's convention that translators move downward.
 Readers who prefer translators to move upward may wish to hold the paper upside-down
 when reading it.}.
Thus the translator equation is
\begin{equation}\label{MCF}
   \vec H = (0,0,-1)^\perp.
\end{equation}
 A {\bf complete translator} is a translator that is complete as a surface in $\RR^3$.
See~\cite{himw-survey} for an introductory survey about complete translators.
In this paper we construct  new examples of complete translators inspired by  old minimal surfaces.\nocite{himw-2}

In  the 19th Century, Scherk discovered a two-parameter family of doubly periodic minimal surfaces. 
(It is one-parameter family if one mods out by homothety.)
An appropriately chosen sequence of these surfaces converges to the helicoid, a singly periodic surface. 
 In this paper, we construct
a two-parameter family of complete translators inspired by Scherk doubly periodic minimal surfaces. 
Some of the examples resemble well-known minimal surfaces
  (see Figures~\ref{fig:scherk-translator} and~\ref{fig:helicoid}),
 but other examples have no minimal surface analog (see  Figures~\ref{fig:scherkenoid} and \ref{fig:pitchforks}).
  Theorem~\ref{translating-scherk-summary} summarizes 
  all the existence and uniqueness results contained in this paper.

Our work depends on an understanding of graphical translators, in particular the classification of complete graphical translators achieved in \cite{himw}. 
We discuss the classification at the end of the introduction.   
Translators can be considered minimal surfaces in a metric conformal to the standard Euclidean metric and this is key to our construction of graphical translators over parallelograms and strips with infinite boundary values. We present this in the following few paragraphs. 
The rest of the paper is concerned with the proof of the existence of the complete translators mentioned above.


\addtocontents{toc}{\protect\setcounter{tocdepth}{0}}
\section*{Translators as Minimal Surfaces}
\label{minimal-section}
Ilmanen~\cite{ilmanen} observed that the defining condition~\eqref{MCF} for a surface 
to translate under mean curvature flow is equivalent to the requirement that the surface be minimal with respect to the metric $g_{ij}=e^{-x_3}\delta_{ij}$.
This allows us to use compactness theorems, curvature estimates, maximum principles,  and monotonicity for $g$-minimal surfaces. 

If $M$ is a {\em graphical} translator  and $\nu$ is  the (upward-pointing) unit-normal vector field to $M$, 
then $\ee_3\cdot \nu$ is a positive  Jacobi field.  Consequently, $M$ is a stable $g$-minimal surface. Therefore a sequence  of translating graphs will converge, subsequentially, to a  translator. 
Moreover the vertical translates of $M$ are also  $g$-minimal and foliate a cylinder $\Omega\times\RR$, where
$\Omega\subset\RR^2$ is the region over which $M$ is a graph. As a consequence, $M$ is a $g$-area minimizing
 surface in $\Omega\times\RR$, and if $\Omega$ is convex, then $M$ is $g$-area minimizing in $\RR^3$
 (See Theorem~\ref{minimizing-theorem} in the appendix.)
  This provides local area estimates (Corollary~\ref{minimizing-corollary}).
 For the surfaces that arise in this paper, the boundaries are polygonal, and we also get curvature estimates up to the boundary (Theorem~\ref{curvature-bound}).
As a consequence, any sequence of such surfaces has a subsequence that converges smoothly away from the corners of the boundary curves; see Remark~\ref{dichotomy}.

 Reflection in  vertical planes and rotations about vertical lines are isometries of the Ilmanen metric. Therefore we can extend $g$-minimal surfaces by $180^\circ$ rotation about vertical lines (Schwarz reflection) and 
 we can use the Alexandrov method of moving planes in our context.



 \section*{ Graphical Translators}
\label{graph-secction}
\addtocontents{toc}{\protect\setcounter{tocdepth}{2}}

A {\bf graphical translator} is a translator that is the graph of a function 
over a domain in $\RR^2$.  In this case, we also refer to the function as a graphical translator.
The simplest example 
is the {\bf grim reaper surface}: it is  the graph  of the function
\begin{equation} \label{GR}
(x,y)\mapsto \log(\sin y)
\end{equation}
over  the strip $\RR \times (0, \pi)$.

If we rotate the grim reaper surface 
about the $y$ axis by an angle $\theta\in (-\pi/2,\pi/2)$ and then dilate by $1/\cos \theta$, 
the resulting surface is  a also a translator.  It is the graph of the function
\begin{equation}\label{TGR}
  (x,y) \mapsto \frac{\log (\sin (y \cos \theta))}{(\cos\theta)^2} + x \tan \theta. 
\end{equation}
over the strip given by $\RR\times(0,\pi/\cos\theta)$.
The graph of~\eqref{TGR}, or any surface obtained from it by translation
and rotation about a vertical axis, is called a {\bf tilted grim reaper  of width $w=\pi/\cos\theta$}.
We can rewrite~\eqref{TGR} in terms of the width $w$ as
\begin{equation}\label{TGR2}
  (x,y) \mapsto (w/\pi)^2 \log (\sin (y (\pi/w)) \pm x \sqrt{(w/\pi)^2 - 1}. 
\end{equation}
(The $\pm$ is there because for $w>\pi$, the tilted grim reaper can tilt in either direction:
the relation  $\cos\theta=\pi/w$ only determines $\theta\in (-\pi/2,\pi/2)$ up to sign.)
The width $w$ can be any number $\ge \pi$.  Note that the grim reaper surface is
the tilted grim reaper with tilt $0$ and  width $\pi$.

Building on fundamental work of X. J. Wang~\cite{wang} and Spruck-Xiao~\cite{spruck-xiao},
we together with Tom Ilmanen fully classified the complete graphical translators in $\RR^3$:

\begin{theorem}[Classification Theorem]\cite{himw}
\label{classification-theorem}
 For every $w>\pi$, there exists (up to translation) a unique complete
translator 
\begin{equation}\label{delta}
   u: \RR \times (0,w)\to\RR.
\end{equation}
for which the Gauss curvature is everywhere $>0$.  The function $u$ is symmetric with
respect to $(x,y)\mapsto (-x,y)$ and $(x,y)\mapsto (x,w-y)$ and thus attains its maximum at $(0,w/2)$.
Up to isometries of $\RR^2$ and vertical translation, the only other complete translating graphs
are the tilted grim reapers
 and the bowl soliton, a strictly convex, rotationally symmetric graph of an entire function.
\end{theorem}

In particular (as Spruck and Xiao had already shown), there are no complete graphical translators
defined over strips of width less than $\pi$.  Moreover
the grim reaper surface is the only example with width $\pi$.

The positively curved translator~\eqref{delta} in the Classification Theorem is called a {\bf $\Delta$-wing}.

\section{The Main Theorem}
\label{main-section}

\newcommand{\garea}{\operatorname{g-area}}

\newcommand{\ray}{\operatorname{ray}}

In this section, we  state the main result of this paper.  To do so,
we need some notation. 

\begin{definition}\label{P-notation}
For $\alpha\in (0,\pi)$, $w\in (0,\infty)$, and $0<L\le \infty$, let
 $  P(\alpha,w,L) $
be the set of points $(x,y)$ in the strip $\RR\times (0,w)$
such that
\[
    \frac{y}{\tan\alpha} < x < L + \frac{y}{\tan\alpha}.
\]
The lower-left corner of the region is at the origin and the interior angle at that corner is $\alpha$.
\end{definition}


\begin{figure}[htbp]
\begin{center}
\includegraphics[width=.55\textwidth]{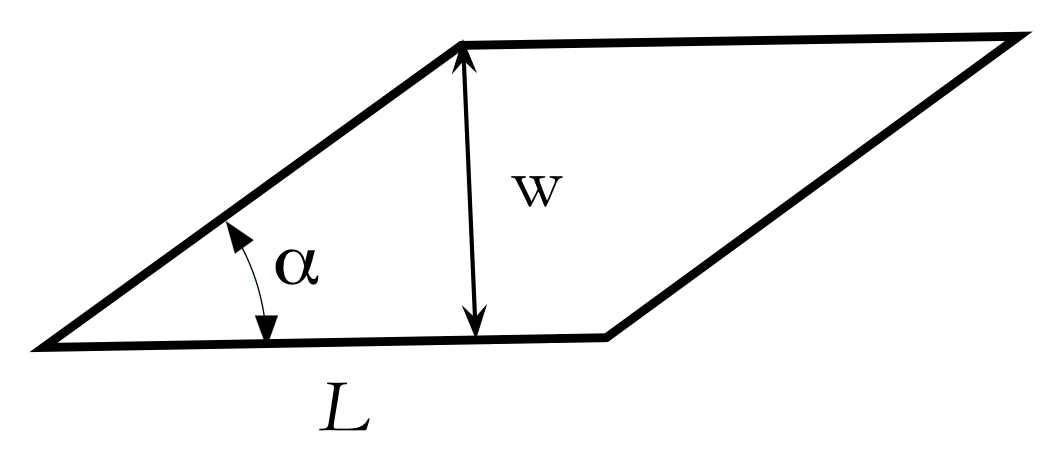}
\caption{\small The parallelogram with base $L$, corner angle $\alpha$, and height $w$.}
\label{fig:para-1}
\end{center}
\end{figure}

If $L<\infty$, $P(\alpha,w,L)$ is a parallelogram with base length $L$
 and width (in the $y$-direction)  $w$ (see Figure~\ref{fig:para-1}.)
The region $P(\alpha,w,\infty)$ may be regarded as a parallelogram whose right-hand side
has drifted off to infinity.

The following theorem summarizes some of the main facts about the classical doubly periodic Scherk minimal surfaces:

\begin{theorem}\label{classical-scherk-summary}
For each $\alpha\in (0,\pi)$, $w\in (0,\infty)$ and $L\in (0,\infty]$, the boundary value problem
\begin{equation*}
\begin{cases}
&u:P=P(\alpha,w,L)\to \RR, \\
&\text{The graph of $u$ is minimal}, \\
&\text{$u= -\infty$ on the horizontal sides of $P$}, \\
&\text{$u=+ \infty$ on the nonhorizontal sides of $P$}
\end{cases}
\end{equation*}
has a solution if and only if $P$ is a rhombus, i.e., if and only if
\[
   L = \frac{w}{\sin\alpha}.
\]
If $P=P(\alpha,w, w/\sin\alpha)$ is a rhombus, then the solution is unique
up to an additive constant, and there is a unique solution $u_{\alpha,w}$
satisfying the additional condition
\[
 \text{$(\cos(\alpha/2), \sin(\alpha/2),0)$ is tangent to the graph of $u$ at the origin}.
\]
The graph of $u_{\alpha,w}$ is bounded by the four vertical lines through the
corners of $P$.  It extends by repeated Schwartz reflection to a
doubly periodic minimal surface $\Ss_{\alpha,w}$.  (See Figure~\ref{fig:scherkfamily}.) 
{As $\alpha\to 0$, the surface $\Ss_{\alpha,w}$
 converges smoothly to the  parallel vertical planes $y=nw$, $n\in \ZZ$.}
 As $\alpha\to \pi$, the { surface}
$\Ss_{\alpha,w}$ converges smoothly to the helicoid given by
\[
  z = x \, \cot \left( \frac{\pi}{w} \, y \right).
\]
\end{theorem}

 \begin{figure}[htbp]
\begin{center}
\includegraphics[height=.19\textheight]{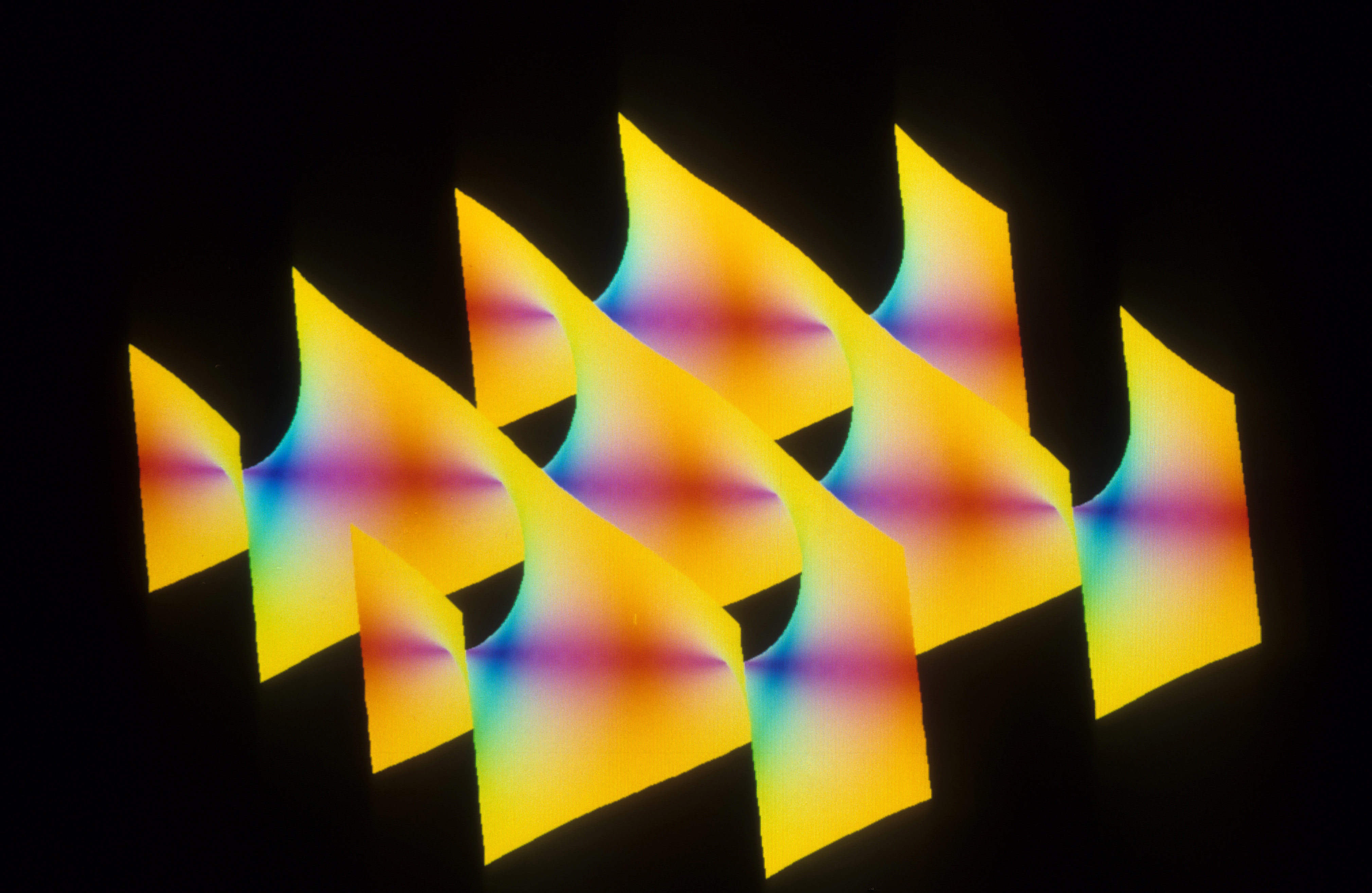}
\includegraphics[height=.19\textheight]{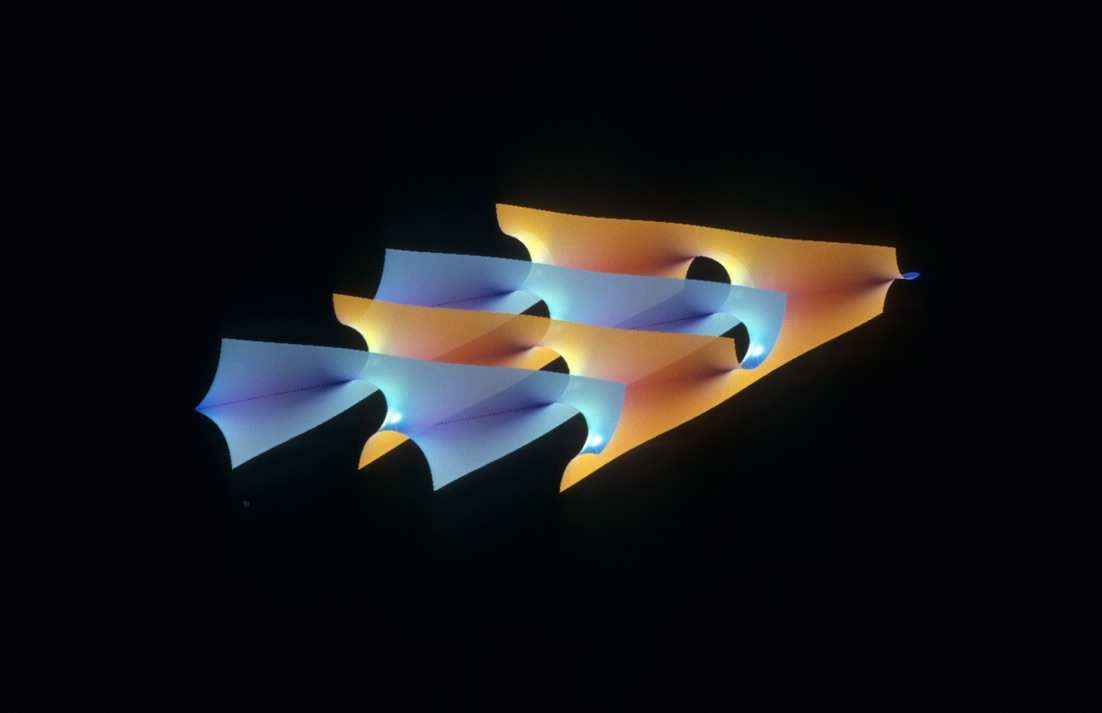}
\caption{The family of Scherk doubly periodic minimal surfaces.}
\label{fig:scherkfamily}
\end{center}
\end{figure}

For translators, 
Theorem~\ref{classical-scherk-summary} has the following analog, 
which summarizes the main results
of this paper.   See also Figure~\ref{fig:scherk-space}.

\begin{theorem}\label{translating-scherk-summary}
For each $\alpha\in (0,\pi)$ and $w\in (0,\infty)$, there is a unique $L=L(\alpha,w)$ in $(0,\infty]$
for which the boundary value problem
\begin{equation*}
\begin{cases}
&u:P=P(\alpha,w,L)\to \RR, \\
&\text{The graph of $u$ is a translator}, \\
&\text{$u= -\infty$ on the horizontal sides of $P$}, \\
&\text{$u=+\infty$ on the nonhorizontal side or sides of $P$}
\end{cases}
\end{equation*}
has a solution.   

The length $L(\alpha,w)$ is finite if and only if $w<\pi$.

If $P=P(\alpha,w, L(\alpha,w))$, then the solution is unique
up to an additive constant, and there is a unique solution $u_{\alpha,w}$
satisfying the additional condition
\[
 \text{$(\cos(\alpha/2), \sin(\alpha/2),0)$ is tangent to the graph of $u$ at the origin}.
\]
The graph of $u_{\alpha,w}$ extends by repeated Schwartz reflection to a
periodic surface $\Ss_{\alpha,w}$.  
If $w<\pi$, then $\Ss_{\alpha,w}$ is doubly periodic and we call it a {\bf Scherk translator}.
If $w\ge \pi$, then $\Ss_{\alpha,w}$ is singly periodic and we call it a {\bf Scherkenoid}.

As $\alpha\to 0$, the surface $\Ss_{\alpha,w}$
 converges smoothly to the parallel vertical planes $y=nw$, $n\in \ZZ$.
 
 As $\alpha\to \pi$, the surface
$\Ss_{\alpha,w}$ converges smoothly, perhaps after passing to a subsequence, to a limit surface $M$.
(We do not know whether the limit depends on the choice of subsequence.)
Furthermore,
\begin{enumerate}[\upshape$\bullet$]
\item
If $w<\pi$, then $M$ is helicoid-like: there is an $\hat{x}=\hat{x}_M\in \RR$ such that $M$ contains the
vertical lines $L_n$ through the points $n(\hat{x},w)$, $n\in\ZZ$.   Furthermore, $M\setminus \cup_nL_n$
projects diffeomorphically onto $\cup_{n\in\ZZ} \{ nw< y< (n+1)w\}$.
\item
If $w> \pi$, then $M$ is a complete, simply connected translator such that $M$ contains $Z$
and such that $M\setminus Z$ projects diffeomorphically onto $\{-w<y<0\}\cup \{0<y<w\}$.
We call such a translator a {\bf pitchfork} of width $w$.
\item
If $w=\pi$, then the component of $M$ containing the origin is a pitchfork $\Psi$ of width $\pi$,
but in this case we do not know
whether $M$ is connected\footnote{If $M$ is not connected, one can show that the other connected
components are grim reaper surfaces.}.
\end{enumerate}
\end{theorem}

\begin{figure}[htbp]
\begin{center}
\includegraphics[height=5cm]{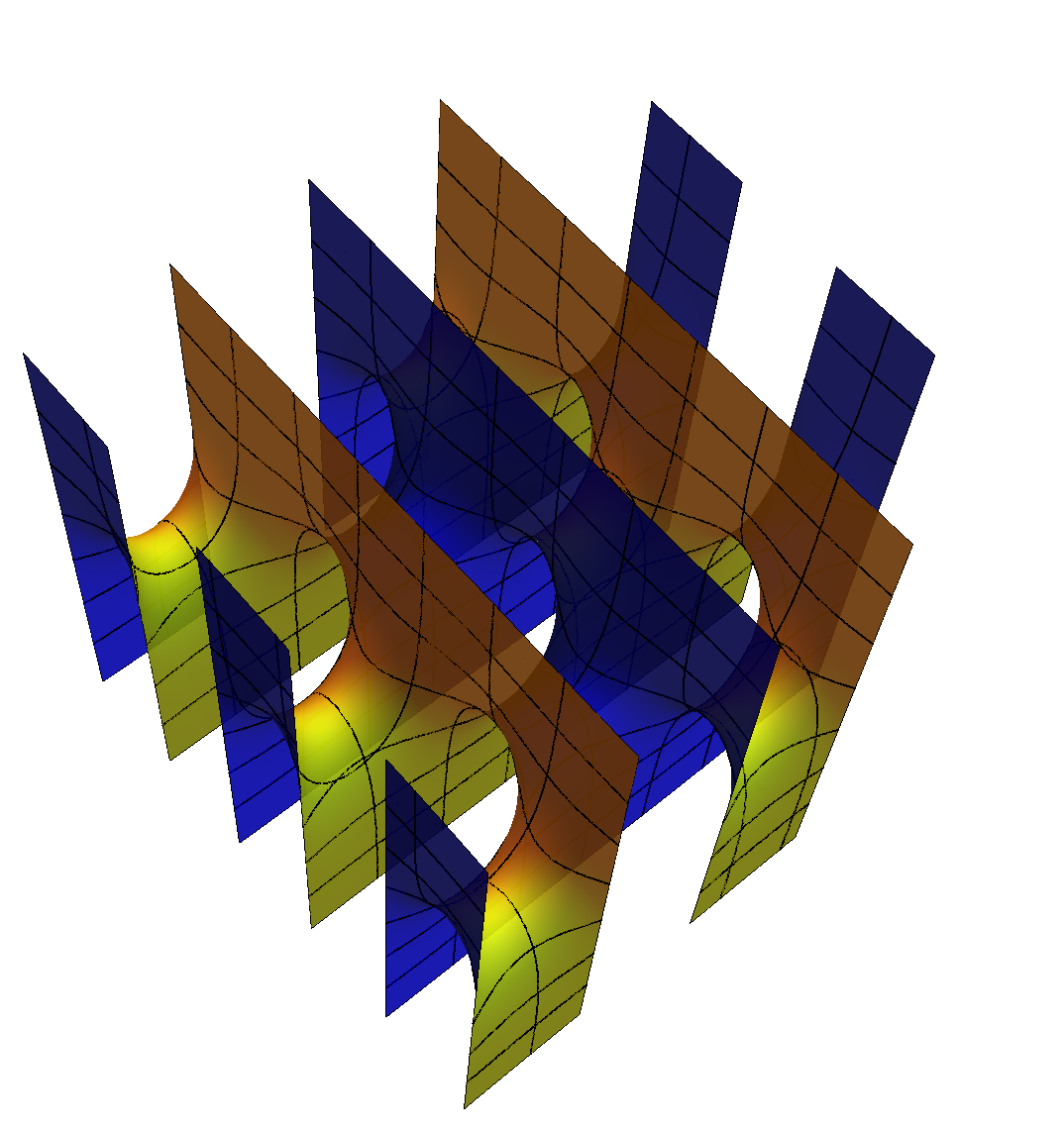}\quad \includegraphics[height=5cm]{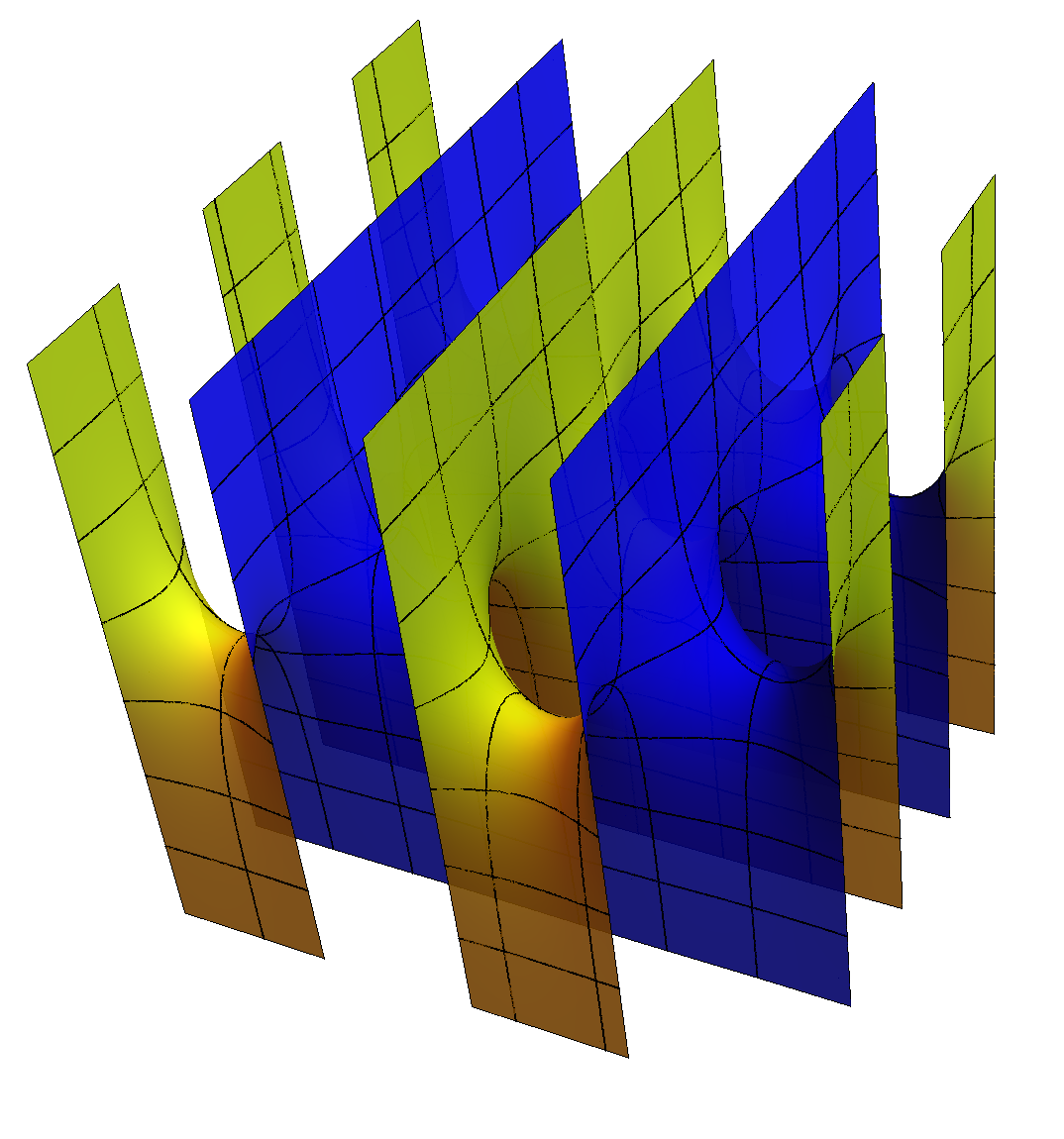}
\caption{The doubly periodic Scherk translator $\mathcal{S}_{\pi/2, \pi/2}$}.
\label{fig:scherk-translator}
\end{center}
\end{figure}

\begin{figure}[htbp]
\begin{center}
\includegraphics[height=5cm]{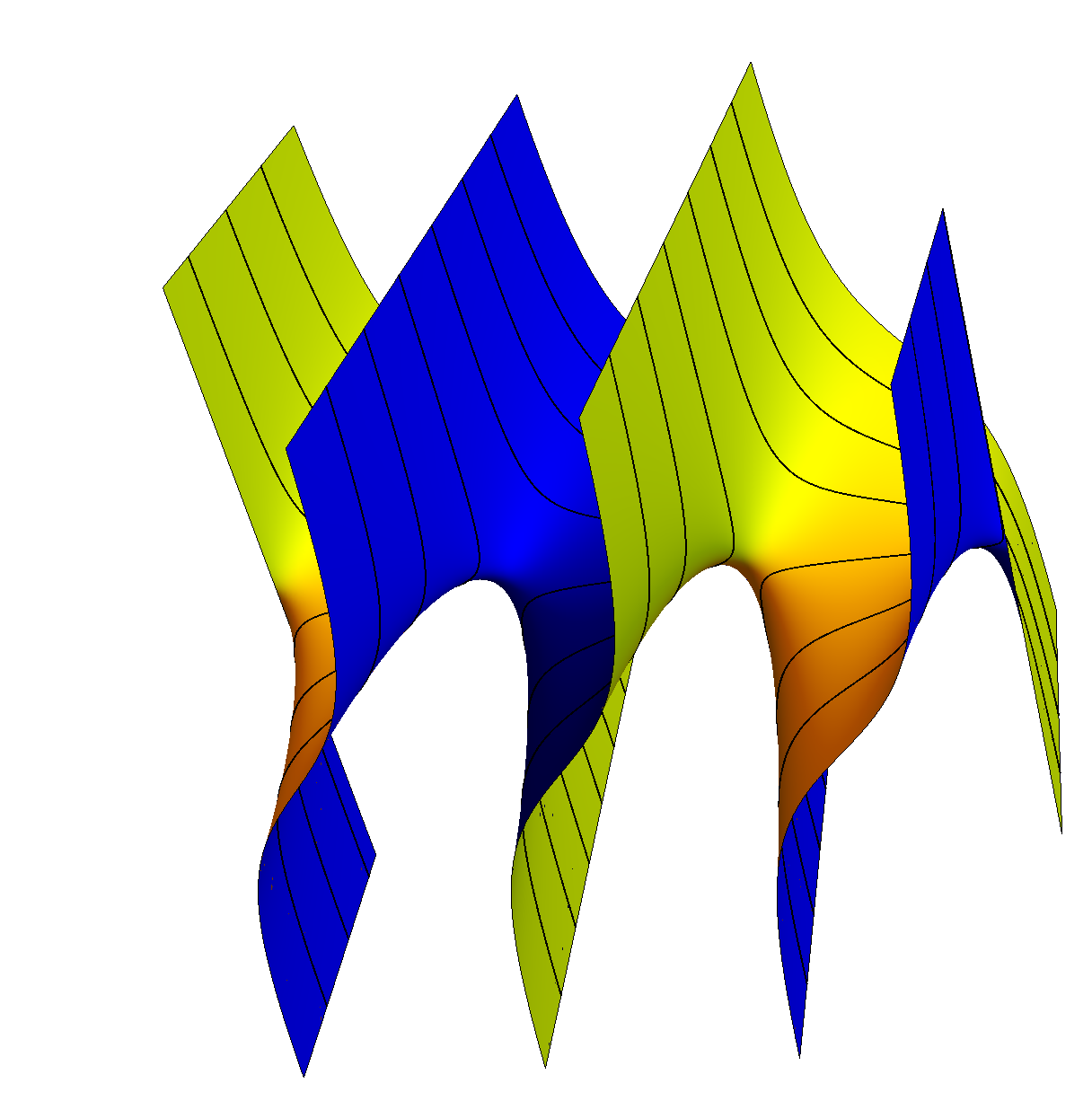}\quad \includegraphics[height=5cm]{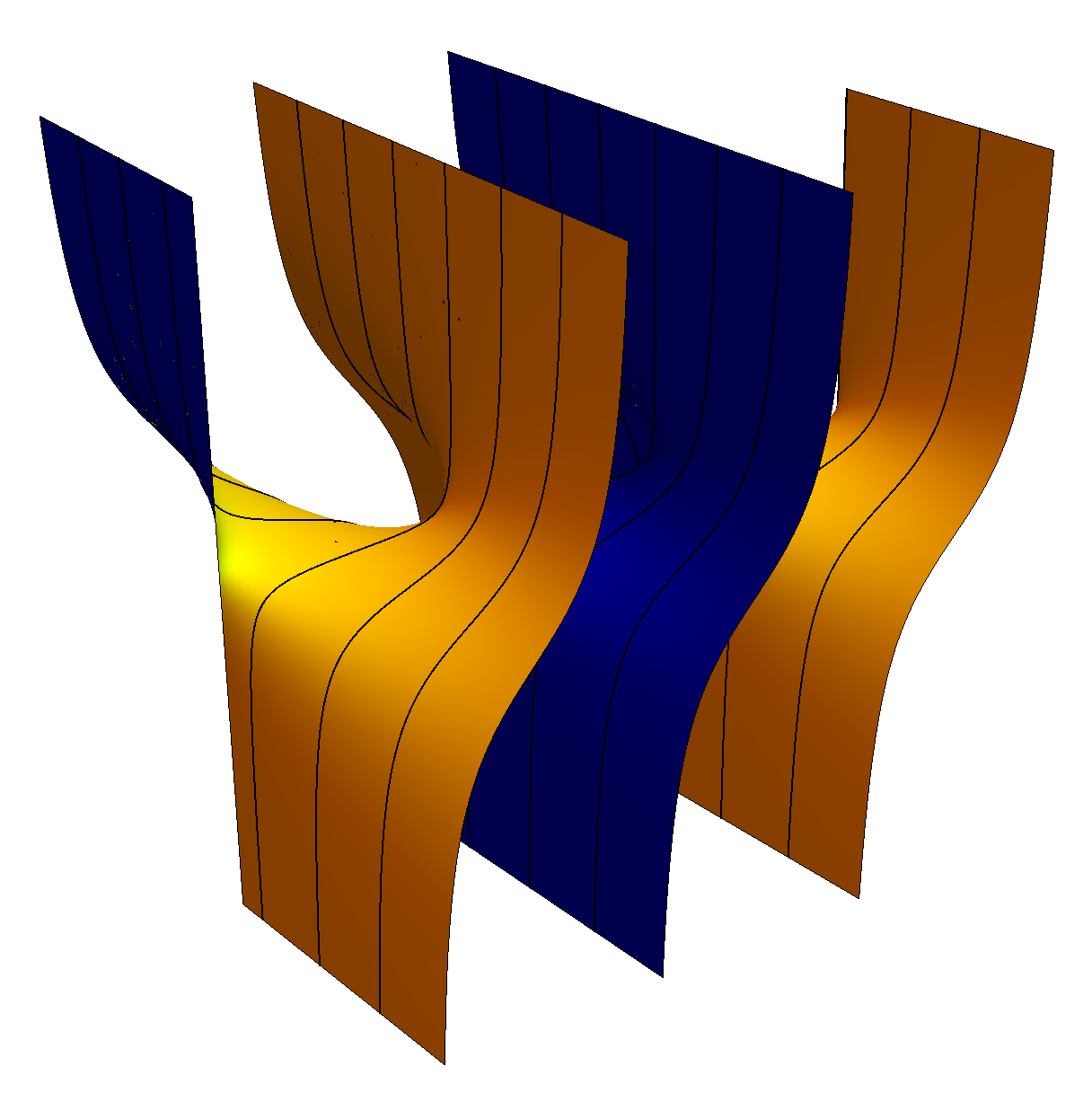}
\caption{A helicoid-like translator of width $w=\pi/2$.}
\label{fig:helicoid}
\end{center}
\end{figure}

The surfaces in Theorem~\ref{translating-scherk-summary} with widths $w<\pi$ are very much like classical Scherk surfaces and helicoids.  See Figures~\ref{fig:scherk-translator} and~\ref{fig:helicoid}.
But the surfaces with width $w\ge \pi$ are very different from standard minimal surfaces:
\begin{enumerate}[$\bullet$]
\item
The Scherkenoid is asymptotic to a infinite family of parallel vertical planes (namely the planes
$y=nw$, $n\in \ZZ$) as $z\to -\infty$, but it is asymptotic to a single vertical plane (the plane
$y=x\tan \alpha$) as $z\to\infty$.  Far away from the plane $y=x \tan\alpha$, the Scherkenoid
looks like a periodic family of tilted grim reapers, each of width $w$.  See Figure~\ref{fig:scherkenoid}.
\item
A pitchfork $M$ of width $w$ is a complete, simply connected, emdedded translator that
lies in the slab $\{|y|< w\}$.   The surface is asymptotic to the plane $y=0$ as $z\to\infty$
and to the three planes $y=0$, $y=w$, and $y= -w$ as $z\to -\infty$.
Away from the $z$-axis, the surface looks like a vertical plane and a tilted grim reaper.
See Figure~\ref{fig:pitchforks}.
\end{enumerate}

\begin{figure}[htbp]
\begin{center}
\includegraphics[height=5cm]{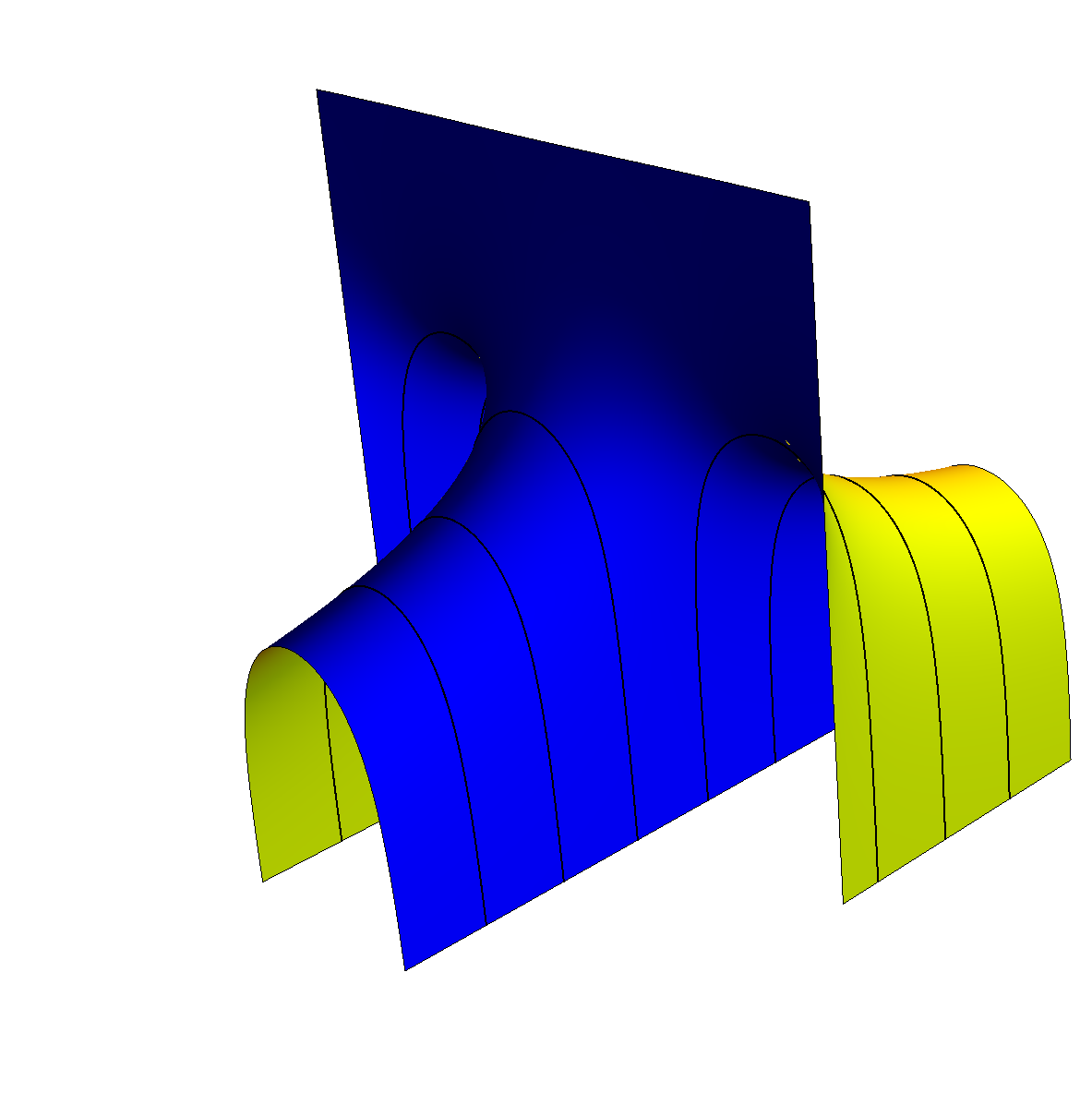} \includegraphics[height=5cm]{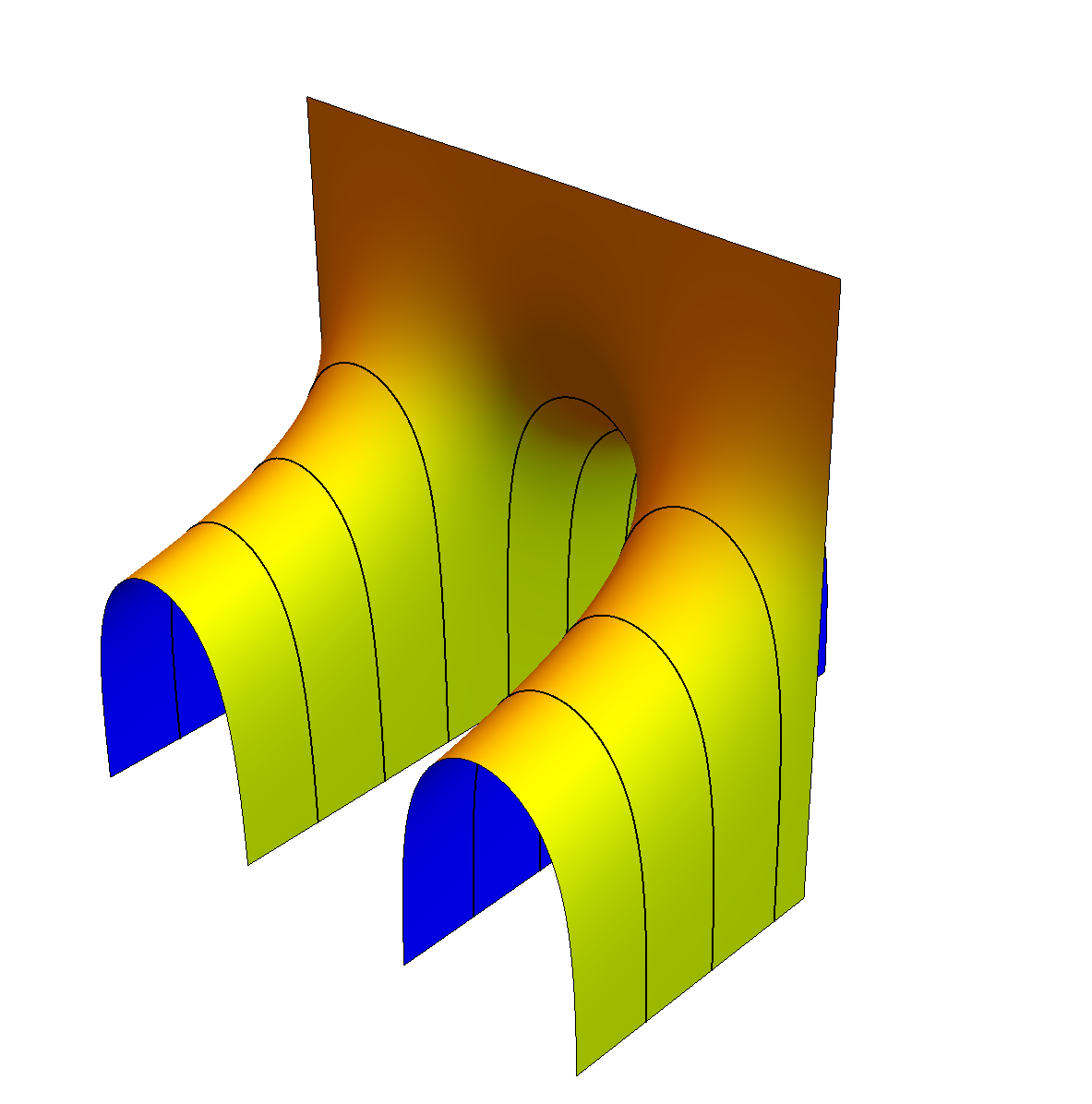}
\caption{The scherkenoid $\mathcal{S}_{\pi/2, \pi}$}.
\label{fig:scherkenoid}
\end{center}
\end{figure}

\begin{figure}[htbp]
\begin{center}
\includegraphics[height=4cm]{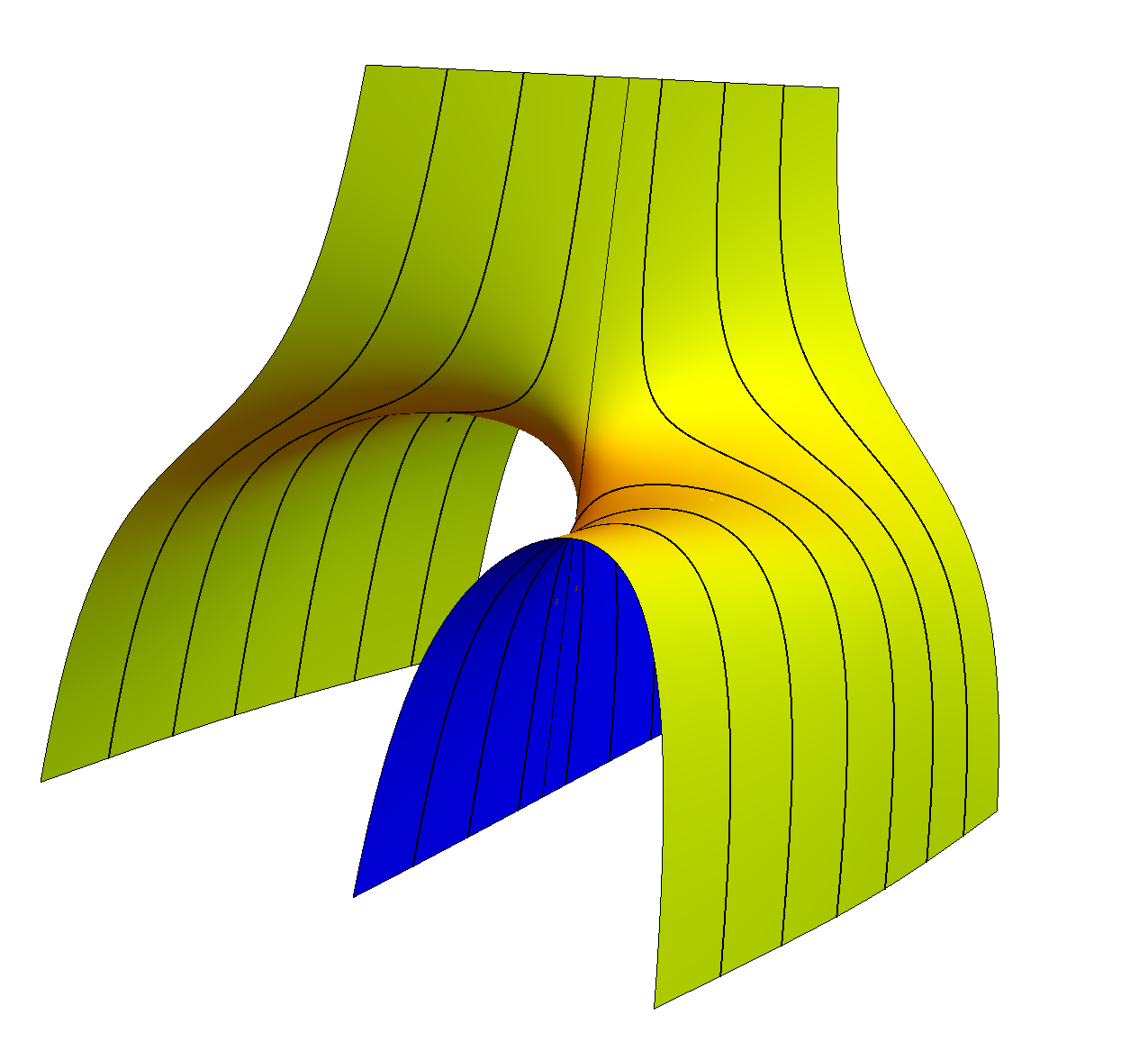} \includegraphics[height=4cm]{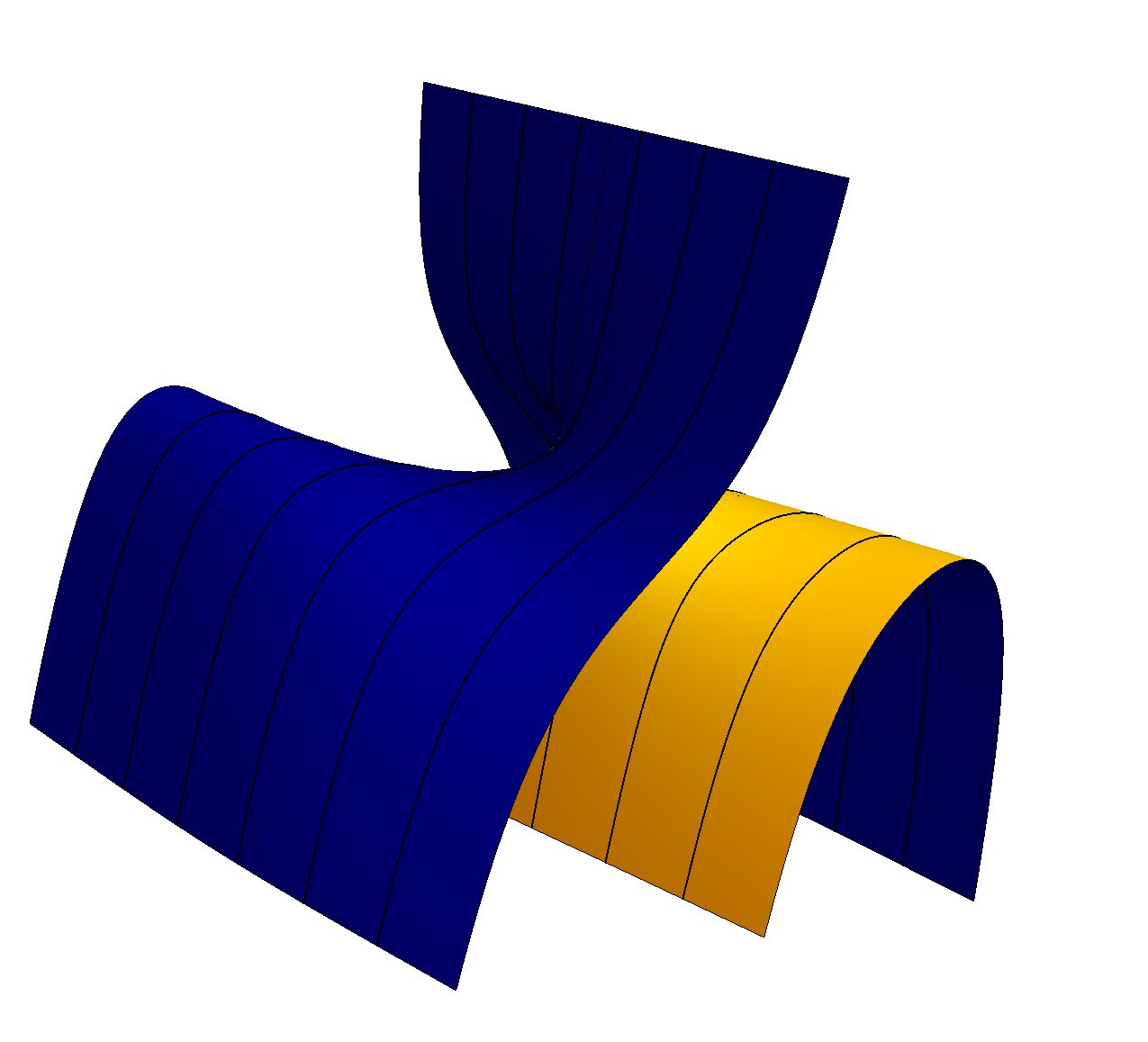}
\caption{A pitchfork  of width $w= \pi$}.
\label{fig:pitchforks}
\end{center}
\end{figure}

\begin{figure}[htbp]
\begin{center}
\includegraphics[height=10.5cm]{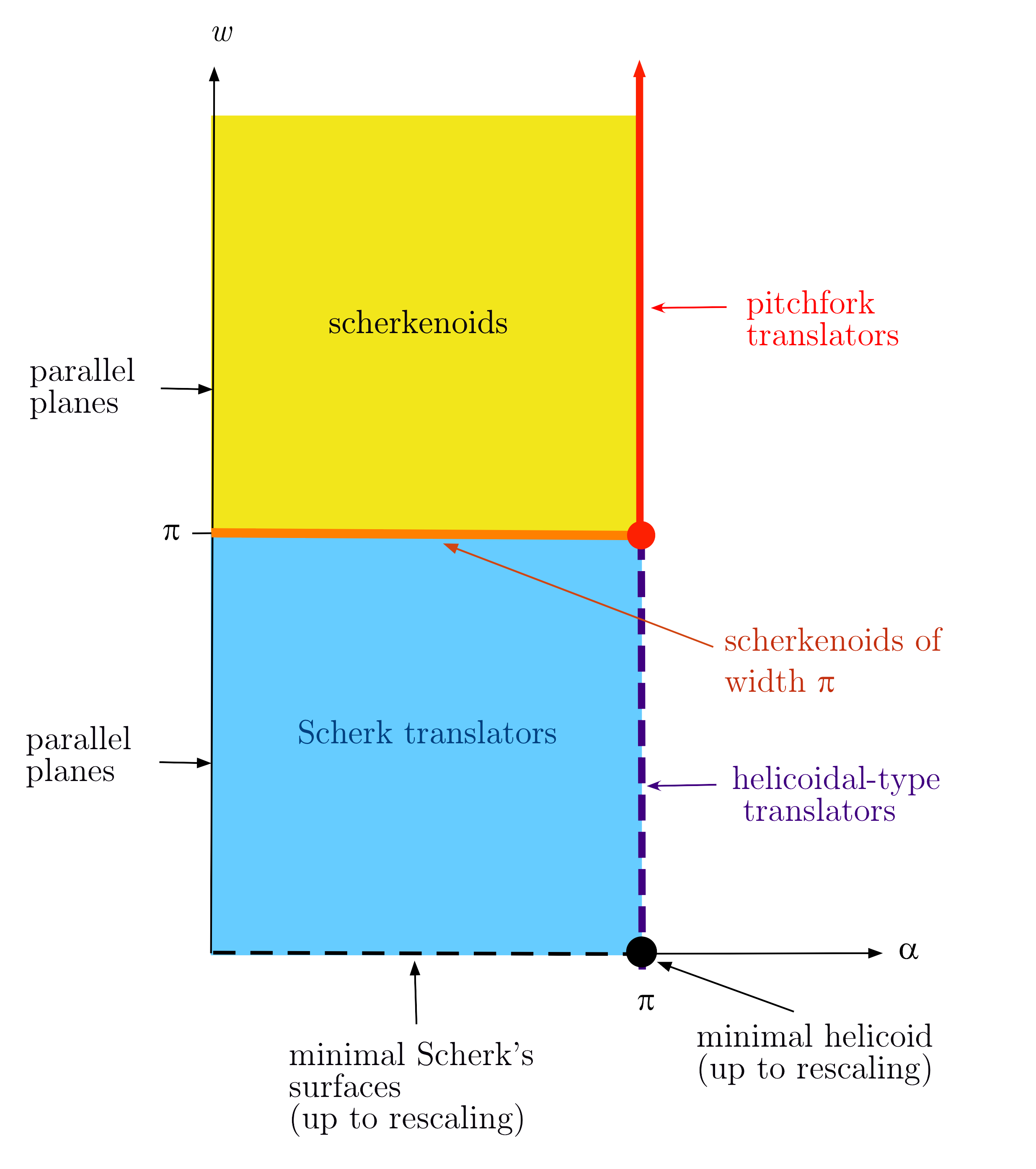}
\caption{
The Scherk translators~(\S\ref{scherk-section}) and
scherkenoids (\S\ref{scherkenoid-section}), and their
limits: translating helicoids (\S\ref{helicoid-section}) and
pitchforks (\S\ref{pitchfork-section}).
}
\label{fig:scherk-space}
\end{center}
\end{figure}

\addtocontents{toc}{\protect\setcounter{tocdepth}{0}}
\section*{Translators as Singularity Models}
\addtocontents{toc}{\protect\setcounter{tocdepth}{2}}

Consider a mean curvature flow $t\in [0,T)\mapsto M(t)$ of smooth closed
surfaces in $\RR^3$.  Suppose that the flow becomes singular at the
spacetime point $(p,T)$, i.e., that 
\[
   \lim_{t\to T} \sup \{ |A(M(t),q)| :   q\in M(t)\cap\BB(p,r)\} \to\infty
\]
for all $r>0$, where $A(\cdot,\cdot)$ is the second fundamental form.
 One can choose $(p_n,t_n)\to (p,t)$ and $\lambda_n\to\infty$ so that
the rescaled flows
\[
   t\in [-\lambda_n^2t_n,0] \mapsto \lambda_n(M(t_n+\lambda_n^{-2}t)-p_n)
\]
converge smoothly to an ancient mean curvature flow 
\[
   t\in (-\infty,0] \mapsto \widehat{M}(t)
\]
with $\widehat{M}(0)$ non-planar.
Determining which ancient flows $\widehat{M}(\cdot)$ arise as blowups in this way is a major unsolved problem.

A special case of the problem is determining which translators arise as blowups.
Any such translator $M$ must have finite topology,  and (by Huisken's mononocity) it must have finite
entropy, which is equivalent to the condition
\[
   \Theta(M):= \sup_{\BB(x,r)\subset \RR^3} \frac{\area(M\cap\BB(x,r))}{\pi r^2} < \infty. \tag{*}
\]
(The entropy of $M$ is bounded above by $\Theta(M)$ and below by a multiple of $\Theta(M)$.)
It follows that the Scherk translators, the Scherkenoids, and
the helicoids constructed in this paper cannot arise as blowups.  
(For those examples, the quantity~\thetag{*} is infinite,  since each is asympotic
as $z\to -\infty$ to an infinite collection of parallel planes.) 
On the other hand,
the pitchforks are simply connected and have finite entropy.
Whether they arise as blowups is an interesting open question.

So far, the only translators known to arise as blowups are the bowl soliton
and the grim reaper surface.  In particular, it is not known whether 
any of the $\Delta$-wings arise as blowups.

Actually, which blowups arise from smooth mean curvature flows depends on whether one considers
embedded or immersed mean curvature flows.  The grim reaper surface arises as a blow up of
an immersed, mean-convex mean curvature 
flow\footnote{Consider a non-embedded closed curve in the the $xz$-plane in $\RR^3$ with nowhere vanishing
curvature.
 If one translates
it in that plane sufficiently far  from $z$-axis, then the resulting surface of revolution
is mean convex and lies inside an Angenent torus.
Under mean curvature flow, it evolves to develop a singularity with a grim reaper surface blowup.},
 but 
not as a blowup of any embedded, mean-convex
mean curvature flow~\cite{white-size}*{Corollary~12.5}.
It has been conjectured that blowups of embedded mean curvature flows
must be ``non-collapsed", which would exclude the pitchforks.  But even if the non-collapsing
conjecture is true, the pitchforks might arise as blowups of smooth immersed mean curvature flows.


\section{Scherk Translators}\label{scherk-section}

 In this section, we  construct translators that are
modeled on  Scherk's doubly periodic minimal surfaces. 
As with minimal surfaces, it suffices to find solutions of the translator equation on a parallelogram
having boundary values $+\infty$ on two opposite sides and $-\infty$ on the other two sides:
repeated Schwarz reflection will then produce a complete, properly embedded, doubly periodic translator.


Recall that  $P(\alpha,w,L)$ is the parallelogram with horizontal sides of length $L$ in the lines $y=0$ and $y=w$,
with the lower-left corner at the origin, and with interior angle $\alpha$ at that corner (Figure~\ref{fig:para-1}).

  \begin{theorem}\label{scherk-translator-theorem} For every $0<\alpha<\pi$  
  and   $0<w<\pi$, there exists a unique $L=L(\alpha,w)>0$
  with the following property:
  there is a smooth surface-with-boundary $\Dd$
  such that $\Dd$ is a translator and such that $\Dd\setminus\partial \Dd$ is the graph
  of  a function 
  \[
      u: P(\alpha,w,L)\to \RR
  \]
 that has boundary values $-\infty$ on the horizontal sides and $+\infty$ on the nonhorizontal sides.

Given $\alpha$, $w$, and $L(\alpha,w)$, the surface $\Dd$ is unique up to vertical translation.
In particular, there is a unique such translator $\Dd=\Dd_{\alpha,w}$ if we also require that
the vector
\[
    (\cos(\alpha/2),\sin(\alpha/2),0)
\]
be tangent to $\Dd$ at the lower-left corner of $P(\alpha,w,L)$.

If $w\ge \pi$ and if $L\in (0,\infty)$, 
there is no translator $u:P(\alpha,w,L)\to\RR$ with the indicated boundary values.
\end{theorem}

\begin{figure}[htbp]
\begin{center}
\includegraphics[width=.36\textwidth]{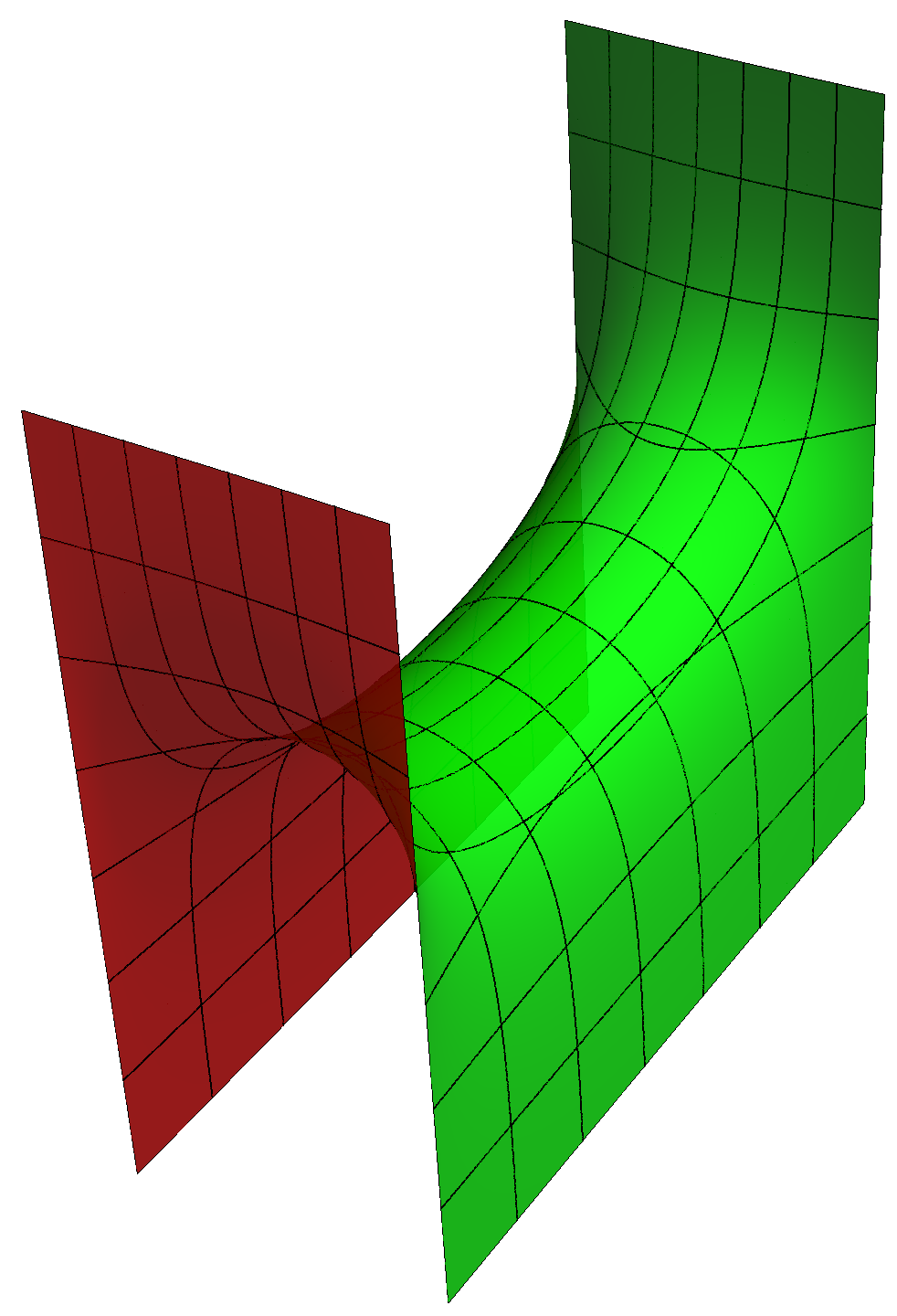}
\caption{The graph of the function $u$ in Theorem \ref{scherk-translator-theorem} for $\alpha=\pi/2$ and $w=\pi/2$. Notice that (contrary to what happens in the minimal case) $L(\pi/2,\pi/2)>\pi/2$.}
\label{prime}
\end{center}
\end{figure}

  \begin{remark}
Since $P(\pi-\alpha,w,L)$ is the image of $P(\alpha,w,L)$ under reflection in the line $x=L/2$,
it follows that $L(\alpha,w)=L(\pi-\alpha,w)$ and that $\Dd_{\pi-\alpha,w}$ 
is the image of $\Dd_{\alpha,w}$
under reflection in the plane $x=L(\alpha,w)/2$, followed by a vertical translation.
  \end{remark}
  
\begin{proof}
Let $P'(\alpha,w,L)$ be $P(\alpha,w,L)$ translated so that its center is at the origin.
Fix a width $w$ in $(0,\pi)$ and an angle $\alpha$ in $(0,\pi)$.
For $L>0$ and $h>0$, let
\[
    u_L^h: P'(\alpha,w,L) \to\RR
\]
be the unique solution of the translator equation that takes boundary values $0$ on the horizontal edges
of the parallelogram and $h$ on the other two edges.  
(See Theorem~\ref{finite-existence} for a proof of existence and uniqueness.)
Let $M_L^h$ be the graph of $u_L^h$.  
The boundary $\Gamma_L^h$ of $M_L^h$ consists of two  segments at height $0$, two 
segments at height $h$, and four vertical segments of length $h$ joining them.   
See Figure~\ref{fig:curve-1}.
By construction, $M_L^h$ is the $g$-area-minimizing disk with boundary $\Gamma_L^h$.
By uniqueness, $M_L^h$ is invariant under rotation by $\pi$ about the $z$-axis, and thus
\begin{equation}\label{center-horizontal}
   Du_L^h(0,0)=0.
\end{equation}

\begin{figure}[htbp]
\begin{center}
\includegraphics[height=5cm]{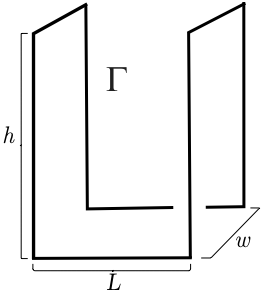} \hspace{0.43in}
\includegraphics[width=.55\textwidth]{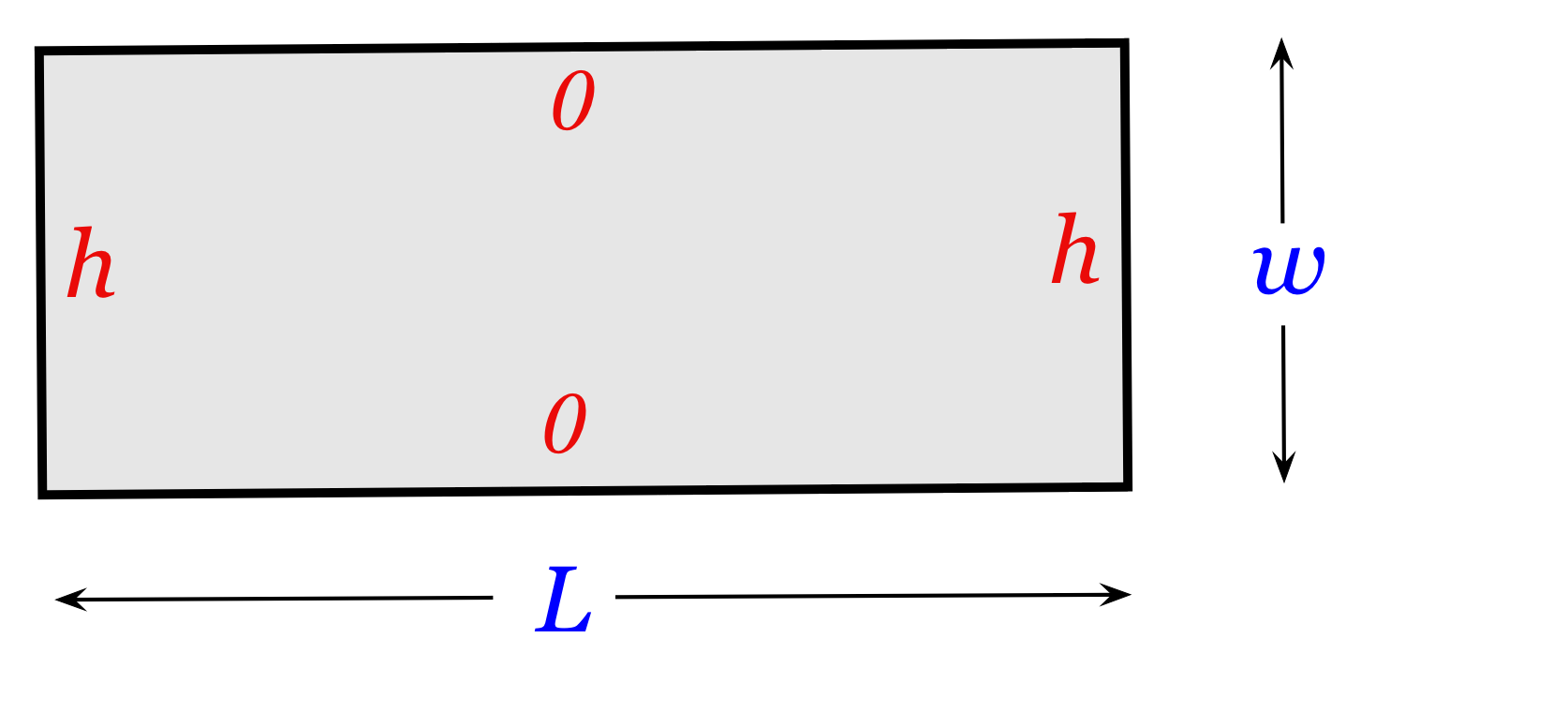}
\caption{Left: The curve $\Gamma$. Right: The  domain with boundary values indicated in red.
The figure shows the case when the parallelogram is a rectangle.}
\label{fig:curve-1}
\end{center}
\end{figure}

As $L\to 0$, the $g$-area of the graph tends to $0$ (since the graph is $g$-area-minimizing).  
Consequently, by the monotonicity formula,
the distance from $(0,0,u_L^h(0,0))$ to $\Gamma_L^h$ tends to $0$.   Thus
\[
    \lim_{L\to 0} u_L^h(0,0) = h.
\]

Next we claim that
\begin{equation}\label{bounded-claim}
  \text{$u_L^h(0,0)$ is bounded above as $L\to\infty$}.
\end{equation}
For suppose to the contrary that there is a sequence $L(i)\to\infty$
for which $u_{L(i)}^h(0,0)\to\infty$.  Then, passing to a further subsequence,
the graph of $u_{L(i)}^h - u_{L(i)}^n(0,0)$ converges smoothly to a complete
translator $M$ lying in the slab $\{|y|\le w/2\}$. 
Since the tangent plane to $M$ at the origin is horizontal, $M$ is a graph.
But by the Classification Theorem~\ref{classification-theorem}, there are no complete translating graphs
in slabs of width less than $\pi$.  This proves~\eqref{bounded-claim}.

Next, we claim that
\begin{equation}\label{exact-limit}
   \lim_{L\to\infty} u_L^h(x,y) = \log(\cos y) - \log(\cos (w/2)).
\end{equation}
To see this, let $M$ be a subsequential limit of the graph of $u_L^h$ as $L\to\infty$.
Note that $M$ is a smooth translator lying in the slab $\{|y|\le w/2\}$
and that $\partial M$ consists of the two lines $\{(x,-w/2,0): x\in \RR\}$ and $\{(x,w/2,0): x\in \RR\}$.
By~\eqref{center-horizontal}, $M$ contains a point on the $z$-axis at which the tangent plane is horizontal.  Thus
$M$ is a graph.
Any translating graph over a strip with boundary values $0$ is translation invariant
in the direction of the strip (\cite{himw}*{Theorem~3.2}), from which~\eqref{exact-limit} follows.

From~\eqref{bounded-claim} and~\eqref{exact-limit}, we see that
for every sufficiently large $h$, there exists an $L(h)$ such that 
\[
    u_{L(h)}^h(0,0)= h/2.
\]

Now let 
\[
    v^h := u_{L(h)}^h - u_{L(h)}^h(0,0)
\]
(see Figure~\ref{fig:curve-2}), and 
consider a sequence of $h$ tending to $\infty$.  
Then there is a subsequence $h_n$ such that $L(h_n)$ converges to a limit $L\in [0,\infty]$
and such that the graphs of the functions $v^{h_n}$ converge smoothly to a limit translator $\Dd$.

\begin{figure}[htbp]
\begin{center}
\includegraphics[height=6cm]{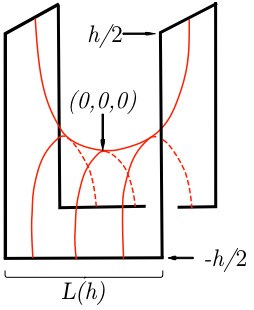}
\caption{The graph of $v^h$}
\label{fig:curve-2}
\end{center}
\end{figure}

Since the tangent plane to $\Dd$ at the origin is horizontal, $\Dd$ is the graph of a function $u$.
Thus $L\ne 0$, since if $L$ were $0$, then $\Dd$ would lie in a vertical plane.
If $L$ were $\infty$, then $\Dd$ would be a complete translating graph in a slab $\{|y|\le w/2\}$
of width less than $\pi$, which is impossible by the Classification Theorem~\ref{classification-theorem}.

Thus $L$ is a finite positive number.  Consequently, $\Dd$ is a translating graph that
lies in $P'(\alpha,w,L)\times\RR$, and the boundary of $\Dd$ consists of the four vertical lines
through the vertices of $P'(\alpha,w,L)$.  By Lemma~\ref{straight-lemma} below,
the domain of $u$ is a convex open subset $\Omega$ of $P'=P'(\alpha,w,L)$, and each component of
\[
   (\partial \Omega)\setminus \{ \textnormal{the corners of $P'$} \}
\]
is a straight line segment.  Since $\Omega$ contains the center of $P'$, 
it follows that $\Omega$ is the entire parallelogram $P'$.
The smooth convergence to $\Dd$ implies that $u$ has the asserted boundary values.
This completes the proof of existence.

The uniqueness assertion is proved in \S\ref{scherk-uniqueness-section}; 
see Theorem~\ref{monotone-theorem}.
The  non-existence for widths $w\ge \pi$ is proved in Proposition~\ref{no-thick-scherks} below.
\end{proof}

\begin{lemma}\label{straight-lemma}
Let $\Omega$ be a convex open subset of $\RR^2$ and let $S$ be a discrete
set of points in $\partial \Omega$.  Let $\Dd$ be a translator that is the graph
of a function $u$ over an open subset $U$ of $\Omega$ and such that $\overline{\Dd}\setminus \Dd$
is contained in the vertical lines $S\times \RR$.  
Then each component of $(\partial U)\setminus S$ is a line, ray, or line segment.
\end{lemma}

\begin{proof}
If $C$ is a connected component of $(\partial U)\setminus S$, then $u=+\infty$
on $C$ or $u=-\infty$ on $C$.  Note that as $z\to\infty$, the surfaces $\Dd+(0,0,z)$
converge smoothly to 
\[
\{p\in \partial \Omega: u(p)=-\infty\}\times \RR.   \tag{*}
\]
Consequently, \thetag{*} is a smooth translator, so each each component of it is flat.
The same reasoning shows that the components of $\partial \Omega$ where $u=\infty$
are also flat.
\end{proof}

  The translator $\mathcal{D}_{\alpha, w}$ in Theorem~\ref{scherk-translator-theorem}
  extends by iterated Schwarz reflection to a complete, doubly periodic translator $\Ss_{\alpha,w}$.
   See Figure~\ref{fig:scherk-translator}.
  These doubly periodic translators are the analogs of the classical doubly periodic Scherk minimal surfaces
   discussed at the beginning of \S\ref{main-section}.

\begin{proposition}\label{no-thick-scherks}
Let $w\ge \pi$ and $L<\infty$.  Then there is no translator
\[
   u: P=P(\alpha,w,L) \to \RR
\]
with boundary values $-\infty$ on the horizontal sides and $+\infty$ on the nonhorizontal sides.
\end{proposition}

\begin{proof}
First suppose that $w>\pi$.
Let $g: \RR\times (w/2- \pi/2, w/2 + \pi/2)\to \RR$ be a grim reaper surface.
If there were such a $u$, then $g-u$ would have a local maximum (since it is $-\infty$ at all boundary
points of its domain), contradicting the strong maximum principle.

Now consider the more delicate case when $w=\pi$.
Suppose there were such a translator $u$.  By adding a constant, we can assume that 
\begin{equation}\label{normalization}
    \min u(\cdot, \pi/2) = 0. 
\end{equation}
Let $M$ be the portion of the graph of $u$ in the region $0\le y \le \pi/2$.
Consider the grim reaper surface 
\[
   g: (x,y) \in \RR\times (0,\pi) \mapsto \log(\sin y).
\]
Let $G$ be the portion of the graph of $g$ in the region $0\le y\le \pi/2$.
We claim that $G$ and $M$ cannot cross each other, i.e., that 
\begin{equation}\label{half-above}
   \text{$u(x,y)\ge g(x,y)$ for $0<y\le \pi/2$}.
\end{equation}
For suppose not.  Then there would be a largest $t>0$ such that the surfaces $G+t\ee_2$ and $M$
have nonempty intersection.  But for that $t$, $G+t\ee_2$ and $M$ would violate the strong maximum principle.
This proves~\eqref{half-above}.

In exactly the same way, one proves that $u(x,y)\ge g(x,y)$ for $\pi/2\le y < \pi$.
Thus $u\ge g$ at all points of $P$.   But by~\eqref{normalization}, we have equality at some point
on the line $y=\pi/2$, violating the strong maximum principle.
\end{proof}


  \section{Uniqueness of Scherk Translators}\label{scherk-uniqueness-section}

  \begin{theorem}[Scherk Uniqueness Theorem]
  \label{scherk-uniqueness-theorem}
Let $P$ be the interior of a parallelogram in $\RR^2$ with two sides parallel to the $x$-axis.
Suppose 
\[
   u, v: P\to \RR
\]
are  translators that have boundary values $-\infty$ on the horizontal sides of $P$ and $+\infty$ on the other sides.
Then $u-v$ is constant.
\end{theorem}

\begin{proof}   
If $\xi$ is a vector in $\RR^2$, let
\[
   u_\xi: P_\xi\to \RR
\]
be the result of translating $u:P\to\RR$ by $\xi$, and let
\begin{align*}
  &\phi_\xi: P\cap P_\xi \to \RR, \\
  &\phi_\xi = u_\xi - v.
\end{align*}

Suppose contrary to the theorem that $u-v$ is not constant.  Then there is a point $p$ where $Du(p)\ne Dv(p)$.
By Theorem~\ref{scherk-gauss-map-theorem}\eqref{gauss-image-assertion},
the image of the Gauss map of the graph of $u$ is the entire upper hemisphere.
Equivalently, $Du(\cdot)$ takes every possible value.  Thus there is a $q\in P$ such that $Du(q)=Dv(p)$.
Let $\xi_0=p-q$.  Then $p$ is a critical point  of $\phi_{\xi_0}$.

We claim that it is an isolated critical point.  For otherwise $\phi_{\xi_0}$ would be constant near $p$
and therefore (by unique continuation) constant throughout $P\cap P_{\xi_0}$, which is impossible.
(Note for example that $\phi_{\xi_0}$ is $+\infty$ on some edges of $P\cap P_{\xi_0}$ and $-\infty$
on other edges.)

It follows (by Corollary~\ref{morse-family-corollary}, for example)
that for every $\xi$ sufficiently close to $\xi_0$, $\phi_\xi$ has a critical point.
In particular, there is such a $\xi$ for which $P_\xi$ is in general position with respect to $P$
(where ``general position" is as in Proposition~\ref{rado-proposition}.)

However, for such a $\xi$, $\phi_\xi$ cannot have a critical point, a contradiction.  Heuristically, it cannot have 
a critical point because $\phi_\xi$ is $+\infty$ on two adjacent sides of its parallelogram domain $P\cap P_\xi$
and $-\infty$ on the other two sides.
Rigorously, it cannot have a critical point by Proposition~\ref{rado-proposition} below.

(Note that $P$ contains exactly one vertex of $P_\xi$, $P_\xi$ contains exactly one vertex of $P$,
and at each of the two points where an edge of $P$ intersects an edge of $P_\xi$, 
one of the functions $u$ and $u_\xi$ is $+\infty$
and the other is $-\infty$.)
\end{proof}

\begin{proposition}\label{rado-proposition}
Suppose that $P_1$ and $P_2$ are convex open planar polygonal regions (perhaps unbounded).
Suppose that that $P_1\cap P_2$ is bounded and nonempty,
and that $P_1$ and $P_2$ are in general position: no vertex of one is in the boundary of the other,
and each intersection of edges is transverse.
Suppose that $M_1$ and $M_2$ are smooth, propertly embedded translators with straight-line boundaries
such that $M_i\setminus\partial M_i$ is the graph of a function $u_i:P_i\to\RR$.  (Thus the boundary
values of $u_i$ are alternately $+\infty$ and $-\infty$ on the edges of $P_i$.)

Let $S$ be the set consisting of the following:
\begin{align*}
&\text{Vertices of $P_1$ that lie in $P_2$}, \\
&\text{Vertices of $P_2$ that lie in $P_1$}, \\
&\text{Points where two $u_i=+\infty$ edges intersect, and} \\
&\text{Points where two $u_i=-\infty$ edges intersect.}
\end{align*}
If $S$ has fewer than four points, then it has exactly two points,
Furthermore, each level set of $u_1-u_2$ is a smooth curve connecting
those two points, and $u_1-u_2$ has no critical points.
\end{proposition}

\begin{proof}
Consider  a level set $Q$ of $u_1-u_2$ in $P_1\cap P_2$.
Since we can add a constant to $u_1$, it suffices to consider the case
when $Q$ is the set where $u_1-u_2=0$.

Note that if $K$ is a suitably chosen large compact subset of $P_1\cap P_2$, then $Q\setminus K$
consists of $\#S$ curves: each point of $S$ is an endpoint of exactly one of those curves.
(This does not involve the translator equation; it follows from transversality.)

Note that $Q$ is a network.  The nodes of $Q$ are the points of $Q$ where $Du_1=Du_2$.
Each node has valence $\ge 4$.  (The valence is $2n$ where $n$ is as in~\eqref{the-form}.)
 The network cannot have any closed loops, since $u_1-u_2$ would have
 a local minimum or local maximum in the interior of the region bounded by such a loop,
 violating the strong maximum principle.
Hence each component of $Q$ is an arc or a tree.  Since the valence of each node is even,
the number of ends of $Q$ is even.

Since $Q$ has an even number of ends and since the ends of $Q$ correspond to points of $S$, 
we see that $S$ has zero or two points.  Since $P_1\cap P_2$ is nonempty, there is a nonempty
level set $Q$, and thus $S$ cannot have zero points.  Thus $S$ has exactly two points.
Consequently $Q$ is a single arc, with no nodes.
\end{proof}

\begin{proposition}\label{4-points-proposition}
Under the hypotheses of Proposition~\ref{rado-proposition}, 
if the set $S$ has at most four points, then $u_1-u_2$ has at most one critical point, and such a
critical point must be nondegenerate.
\end{proposition}

\begin{proof}
We may assume that $S$ has exactly four points, as otherwise the assertion follows from
Proposition~\ref{rado-proposition}.
Exactly as in the proof of Proposition~\ref{rado-proposition},
each level set is a tree or union of trees with endpoints in $S$.
Since each node has valence $\ge 4$, each level set must be one of the following:
\begin{enumerate}[\upshape (1)]
\item A pair of arcs, or
\item A  tree with exactly one node, a node of valence $4$.
\end{enumerate}
In case 2, the node is a nondegenerate critical point of $u_1-u_2$.

By elementary topology, any two trees in $P_1\cap P_2$ whose
endpoints are the four points of $S$ must intersect each other.  Since different
level sets are disjoint, this means that at most one level set can have a node.
\end{proof}

\begin{theorem}\label{monotone-theorem}
Suppose for $i=1,2$ that
\[
u_i:  P(\alpha, w_i,L_i) \to \RR
\]
is a graphical translator such that $u_i$ has boundary value $-\infty$ on the horizontal
edges of its domain and $+\infty$ on the nonhorizontal edges.
Suppose also that $w_1\le w_2$ and that $L_1\ge L_2$.   Then $w_1=w_2$, $L_1=L_2$,
and $u_1-u_2$ is constant.
\end{theorem}

\begin{proof}
The proof is almost identical to the proof of Theorem~\ref{scherk-uniqueness-theorem}, so we omit it.
\end{proof}

Theorem~\ref{monotone-theorem} implies that for each $\alpha$, 
the length $L(\alpha,w)$ is a strictly increasing function of $w$.


\section{The behavior of $L(\alpha,w)$}\label{behavior-section}

Recall that (according to Theorem~\ref{scherk-translator-theorem}) for each $(\alpha,w)\in (0,\pi)^2$, there is a unique $L=L(\alpha,w)$
for which  there is a translator $u:P(\alpha,w,L)\to\RR$ with boundary
values $-\infty$ on the horizontal sides and $+\infty$ on the nonhorizontal sides.

In~\cite{JenkSerr-II}, Jenkins and Serrin proved that a parallelogram
can be the fundamental domain of a doubly periodic Sherk-type minimal
surface if and only if it is a rhombus.   
Note that the length of the nonhorizontal side of $P(\alpha,w,L)$ is $w/\sin\alpha$, so $P(\alpha,w,L)$
is a rhombus if and only if $L=w/\sin\alpha$.
In the case of translators, the Jenkins-Serrin
proof gives an inequality:

\begin{theorem}\label{jenkins-serrin-like-theorem}
For $0<\alpha<\pi$, $0<w<\pi$, and $L=L(\alpha,w)$, 
\begin{equation}\label{non-rhombus}
    0 <L -  \frac{w}{\sin\alpha}  < \frac12 wL,
\end{equation}
\end{theorem}

Note that $w/\sin\alpha$ is the length of the nonhorizontal side,
so the inequality~\eqref{non-rhombus} implies that the parallelogram
$P(\alpha,w,L(\alpha,w))$ is never a rhombus,
though it is very close to one when $w$ is small:

\begin{corollary}
\[
  \lim_{w\to 0} \frac{w}{L(\alpha,w)} = \sin\alpha.
\]
\end{corollary}

\begin{proof}[Proof of Theorem~\ref{jenkins-serrin-like-theorem}]
Let $u:P(\alpha,w,L)\to \RR$ be a translator that is $-\infty$ on the horizontal sides and $\infty$ on the
nonhorizontal sides.  Then $u$ satisfies the translator equation:
\[
  -\Div\left(\frac{\nabla u}{\sqrt{1+|\nabla u|^2}} \right) = \frac{1}{\sqrt{1+|\nabla u|^2}}. \tag{*}
\]
Integrating over the parallegram $P$ and using the divergence theorem gives
\begin{equation}\label{js-equality}
   2L - 2\frac{w}{\sin\alpha} = \iint_P (1+ |\nabla u|^2)^{-1/2}\,dx\,dy \le \area(P) = wL
\end{equation}
The inequalities in Theorem~\ref{jenkins-serrin-like-theorem} follow immediately.
\end{proof}

\begin{theorem}\label{scherk-convergence-theorem}
Let $\Dd_{\alpha,w}$ and $L(\alpha,w)$ be as in Theorem~\ref{scherk-translator-theorem}.  Then
\begin{enumerate}[\upshape(1)]
\item\label{L-continuous-item}  $L(\alpha,w)$ and $\Dd_{\alpha,w}$ depend
  continuously on $(\alpha,w)\in (0,\pi)^2$.
\item\label{L-long-item} If $\alpha_i\in (0,\pi)$ converges to $\alpha\in (0,\pi)$ and if $w_i\in (0,\pi)$ converges to $\pi$,
then $L(\alpha_i,w_i)\to\infty$.  
\end{enumerate}
\end{theorem}

\begin{proof}
Suppose $(\alpha_i,w_i)$ in $(0,\pi)^2$ converges to $(\alpha,w)$ in $(0,\pi)^2$.
By Theorem~\ref{jenkins-serrin-like-theorem},
\begin{equation*}
    \frac{w_i}{\sin\alpha_i}  <L_i,
\end{equation*}
where $L_i=L(\alpha_i,w_i)$, so, after passing to a subsequence, we can assume that the $L_i$
converge to a limit $L$ with 
\[
    0 < w/\sin\alpha\le L \le \infty.
\]

Let $\Dd_i=\Dd_{\alpha_i,w_i}$.   Let $p_i$ be the center of the parallelogram $P_i:=P(\alpha_i,w_i,L_i)$
and let $(p_i,z_i)$ be the corresponding point in $\Dd_i$.
If $L=\infty$, then $\Dd_i-(p_i,z_i)$ would converge subsequentially to 
a complete translator $M$ in $\RR\times [-w/2,w/2]$.  Since the tangent plane to $M$ at the origin
is horizontal, $M$ would be a graph, which is impossible since $w<\pi$.
Thus $L<\infty$, so the parallelograms $P_i$ converge
to the parallelogram $P=P(\alpha,w,L)$.

Let $\Dd_i'=\Dd_i-(0,0,z_i)$.  Note that the tangent plane to $\Dd_i'$ at $(p_i,0)$ is horizontal.
By passing to a subsequence, we can assume that $\Dd_i$ and $\Dd_i'$ converge
smoothly to limits $\Dd$ and $\Dd'$.

Let $p$ be the center of the parallelogram $P=P(\alpha,w,L)$.
Since the tangent plane to $\Dd'$ at $(p,0)$ is horizontal, the component $M'$ of
$\Dd'\setminus \partial \Dd'$ containing $(p,0)$ is the graph of a function
$u':\Omega\to \RR$.  By Lemma~\ref{straight-lemma}, $\Omega$ is a convex
open subset of $P$ containing the center $p$, and each component of 
\[
  (\partial \Omega)\setminus \{\textnormal{the corners of $P$}\}
\]
is a straight line segment.  It follows that the domain of $u'$ all of $P$.  

Since the vector $(\cos(\alpha/2),\sin(\alpha/2),0)$ is tangent to $\Dd$ 
at the origin (the lower-left corner of $P$),
we see that $\Dd$ contains points in the interior of $P\times\RR$.
If $\Dd$ were not a vertical translate of $\Dd'$, then we could find a vertical translate
$\Dd'+(0,0,c)$ of $\Dd'$ that has nonempty, transverse intersection with $\Dd$.
But then for large $i$, $\Dd_i'+(0,0,c)$ would intersect $\Dd_i$ transversely, which is impossible
(since $\Dd_i'$ is a vertical translate of $\Dd_i$).
Thus $\Dd$ is a vertical translate of $\Dd'$, so $\Dd\setminus \partial\Dd$ is the graph of a function $u:P\to \RR$.

By smooth convergence, 
\begin{enumerate}[\upshape$\bullet$]
\item The boundary of $\Dd$ is the four vertical lines through the corners of $P$.
\item The boundary values of $u$ are $-\infty$ on the horizontal sides and $+\infty$ on the nonhorizontal sides.
\item The vector $(\cos(\alpha/2), \sin(\alpha/2),0)$ is tangent to $\Dd$ at the origin.
\end{enumerate}
It follows that $L=L(\alpha,w)$ and $\Dd=\Dd_{\alpha,w}$.  This proves Assertion~\ref{L-continuous-item}.

To prove Assertion~\ref{L-long-item}, suppose it were false.  Then after passing to a subsequence, 
 the $L_i:=L(\alpha_i,w_i)$ converge to a finite limit $L$.  
 Exactly as in the proof of Assertion~\ref{L-continuous-item},
we get existence of a translator
\[
     u: P(\alpha,\pi,L)\to\RR
\]
with boundary values $-\infty$ on the horizontal sides and $+\infty$ on the nonhorizontal sides.
But, according to Theorem~\ref{scherk-translator-theorem}, no such translator exists .  
This proves Assertion~\ref{L-long-item}.
\end{proof}

\begin{corollary}\label{scherk-convergence-corollary}
As in Theorem~\ref{scherk-translator-theorem}, let $\Ss_{\alpha,w}$ be the doubly periodic translator
obtained from $\Dd_{\alpha,w}$ be repeated Schwarz reflection.  Then $\Ss_{\alpha,w}$
depends continuously on $(\alpha,w)\in (0,\pi)^2$.
\end{corollary}


\section{Curvature and the Gauss Map}\label{curvature-section}

In this section we prove that the surfaces constructed in this paper
 all have strictly negative curvature. We begin with the Scherk translator.

\begin{theorem}\label{scherk-gauss-map-theorem}
Let $\Dd=\Dd_{\alpha,w}$ be a fundamental piece of a Scherk translator, 
as in Theorem~\ref{scherk-translator-theorem}.
Then 
\begin{enumerate}[\upshape(1)]
\item\label{negative-assertion} $\Dd$ has negative Gauss curvature everywhere, and
\item\label{gauss-image-assertion}  the Gauss map is a diffeomorphism
from $\Dd$ onto 
\[
   \SS^{2+} \setminus Q
\]
where $\SS^{2+}$ is the closed upper hemisphere and 
where $Q$ is the set consisting of the four unit vectors in the equator $\partial \SS^{2+}$
that are perpendicular to the edges of the parallelogram $P$ over which $\Dd$ is a graph.
\end{enumerate}
\end{theorem}

Thus $Q$ consists of the four vectors $\pm\ee_2$ and $\pm (-\sin\alpha,\cos\alpha,0)$.
It follows from the theorem that the Gauss map from the doubly periodic surface $\Ss_{\alpha,w}$
may be regarded as covering space of $\SS^2\setminus Q$.
It is not a universal cover because $\Ss_{\alpha,w}$ is not simply connected. (Indeed, it has
infinite genus.)

\begin{proof}[Proof of Theorem~\ref{scherk-gauss-map-theorem}]
By the translator equation $\vec H= -(0,0,1)^\perp$, and because the boundary of $\Dd$ 
consists of vertical lines, $\Dd$ has mean curvature $0$ at all boundary points,
and therefore Gauss curvature is $\le 0$ at all boundary points.  Note that $\nu\cdot\ee_3$ is a Jacobi
field on $\Dd$ that is positive everywhere on $\Dd\setminus \partial \Dd$.
Hence by the Hopf Boundary Lemma,  its derivative is nonzero at each boundary point.  
Thus the second fundamental form
is nonzero, so the Gauss is strictly negative.
Thus $K$ is negative everywhere on $\partial \Dd$.  In Corollary~\ref{K-not-zero} below, we show that 
there are no interior points where $K=0$.  Hence $K<0$ everywhere, since $\Dd$ is connected.

Now we prove Assertion~\ref{gauss-image-assertion}.
Compactify $\Dd$ by adding four points at infinity corresponding to the four ends of $\Dd$.
Let $\widehat{\Dd}$ be the compactified surface.  Then $\widehat{\Dd}$ is topologically a closed
disk.  Note that the Gauss map $\nu$ extends continuously to $\widehat{\Dd}$ (because $\nu$
has a limiting value as we go to infinity along an end.)  The Gauss map takes
those four points at infinity to the four points of the set $Q$ in the statement of the theorem.

Note also that
\begin{enumerate}
\item $\nu$ is a continuous map from $\widehat{\Dd}$ to the closed upper hemisphere of $\SS^2$.
\item $\nu$ is locally a smooth diffeomorphism at all interior points of $\widehat{\Dd}$.
\item $\nu$ maps $\partial \widehat{\Dd}$ homeomorphically onto the equator.
\end{enumerate}
It follows by elementary topology that $\nu$ maps $\widehat{\Dd}$ homeomorphically onto 
the upper hemisphere.  Furthermore, since the Gauss curvature is negative everywhere on $\Dd$,
the map is a smooth diffeomorphism except at the four points of $\widehat{\Dd}\setminus \Dd$.
\end{proof}


\begin{remark}\label{trivial-remark}
Let $I$ be a closed horizontal segment, ray, or line.
Let $S$ be an open strip such that $\partial S$ does not contain any endpoint of $I$.
If $I$ is a ray or line, assume that $S$ is not horizontal.  Note that if $S\cap I$ is nonempty,
then the number of endpoints of $I$ in $S$ plus the number of points in $(\partial S)\cap I$
is $2$.
\end{remark}

\begin{proposition}\label{reaper-count-proposition}
Let $P=P(\alpha,w,L(\alpha,w))$ and let $u:P\to \RR$ be a Scherk translator, i.e., a translator
that has boundary values $-\infty$ on the horizontal sides and $+\infty$ on the nonhorizontal sides.
 Let $I_1$ and $I_2$ be the horizontal edges of $P$.
Suppose $g:S\to\RR$ is a tilted grim reaper over an open strip~$S$.
Then $g-u$ has at most one critical point, and such a critical point must be non-degenerate.
If $S\cap I_1$ or $S\cap I_2$ is empty, then $g-u$ has no critical points.
\end{proposition}

Both assertions of the conclusion can be combined into one statement:
\[
N(g-u) 
\le 
\#\textnormal{(components of $S\cap (I_1\cup I_2)$)} - 1,
\]
where $N(g-u)$ denotes number of critical points of $g-u$, counting multiplicity.
(See~\S\ref{morse-section} for the definition of multiplicity.)

\begin{proof}
{\bf Case 1:} $\partial S$ does not contain any vertices of $P$.
In this case, the proposition follows immediately from Remark~\ref{trivial-remark} together
with Propositions~\ref{rado-proposition} and~\ref{4-points-proposition}.

{\bf Case 2:} General $S$.
Let $p_0$ be a point in the strip $\{(x,y): 0<y<w\}$ that lies to the left of $P$.
For $|\theta|<\pi$, 
let $g_\theta:S_\theta \to \RR$ be the result of rotating $g:S\to\RR$ counterclockwise about $p_0$
 by angle $\theta$.  For all sufficiently small $|\theta|>0$, $\partial S_\theta$ will not contain any vertex
 of $P$.  
By Case 1, $N(g_\theta-u)\le 1$ for such $\theta$.
It follows (see Corollary~\ref{morse-family-corollary}) that $N(g_0-u)\le 1$.
This proves the first assertion of the proposition.

Now suppose that $S\cap I_1$ is empty (where $I_1$ is the lower horizontal edge of $P$).  
Then $S_\theta\cap I_1$ is empty for all small $\theta>0$,
so  $N(g_\theta-u)=0$ by Case 1.  Thus $N(g_0-u)=0$ by Corollary~\ref{morse-family-corollary}.

Similarly, if $S\cap I_2$ is empty, then $N(g_\theta-u)=0$ for all $\theta<0$ sufficiently close to $0$, 
and therefore $N(g_\theta-u)=0$ by Corollary~\ref{morse-family-corollary}.
\end{proof}

\begin{corollary}\label{K-not-zero} There is no point $p$ in $\Dd\setminus\partial\Dd$ at which the Gauss
curvature is zero.
\end{corollary}

\begin{proof}
The corollary is an immediate consequence of 
Proposition~\ref{reaper-count-proposition} and the following observation: 
Suppose that $u:\Omega\to\RR$ is any graphical translator (not necessarily complete)
over an open set $\Omega$
 and that $(x_0,y_0)\in \Omega$ is a point
for which the Gauss curvature of the graph of $u$ is $0$.
   Then one can find a tilted grim reaper $g:S\to\RR$
such that $u$ and $g$ make contact of order $2$ or higher at $(x_0,y_0)$.
(If this is not clear, see Remark~\ref{matching} below.)
\end{proof}

\begin{remark}\label{matching}
Suppose that $u:\Omega\to \RR$ is a translator and that 
\[
 \det D^2u(x_0,y_0)=0.
\]
By translating and rotating, we can assume that $(x_0,y_0)=(0,0)$ and 
that $\ee_1$ is an eigenvector of $D^2u(0,0)$ with 
eigenvalue $0$.  Note that if $g$ is the tilted grim reaper~\eqref{TGR},
then $D_1g\equiv \tan\theta$.  Thus we can choose $\theta$ so that $D_1g\equiv D_1u(0,0)$.
As $y$ goes from $-w/2$ to $w/2$, $D_2g(0,y)$ goes from $\infty$ to $-\infty$, so there exists a $\hat{y}$
for which $D_2g(0,\hat{y})=D_2u(0,0)$.

Thus 
\begin{align*}
D g(0,\hat{y}) &= Du(0,0), \\
D_{11}g(0,\hat{y}) &=0 = D_{11}u(0,0), \, \text{and} \\
D_{12}g(0,\hat{y}) &=0= D_{12} u(0,0).
\end{align*}
By the translator equation, it follows that $D_{22}g(0,\hat{y})=D_{22}u(0,0)$.
Hence we can translate $g$ and add a constant to get a tilted grim reaper $\tilde g$ that makes
contact of order $\ge 2$ with $u$ at $(0,0)$.
\end{remark}

\addtocontents{toc}{\protect\setcounter{tocdepth}{0}}
\section*{Other Domains}
\addtocontents{toc}{\protect\setcounter{tocdepth}{2}}

The arguments above in this section also work for other domains that will occur later in this paper.

\begin{theorem}\label{other-negative-theorem}
Let $w>0$.
Let $I_1$ be the ray $(0,\infty)\times \{0\}$.   Let $I_2$ be one of the following:
\begin{align*}
&\text{A ray of the form $(\hat{x},\infty)\times \{w\}$, or} \\
&\text{A ray of the form $(-\infty,\hat{x})\times \{w\}$, or} \\
&\text{The line $\RR\times\{w\}$.}
\end{align*}
Let $\Omega$ be the interior of the convex hull of $I_1\cup I_2$.
Let $M$ be a smooth translator such that $M\setminus \partial M$ is the graph of a function
\[
   u: \Omega \to \RR
\]
with boundary values $-\infty$ on $I_1\cup I_2$ and $+\infty$ on 
$J:=(\partial \Omega)\setminus \overline{I_1\cup I_2}$.
Then 
 the Gauss curvature is $<0$ at all points of $M$.
\end{theorem}

(If $I_2$ is a line or a leftward-pointing ray, then $\Omega$ is the strip $\RR\times (0,w)$.
If $I_2$ is a rightward-pointing ray, then $\Omega$ is the portion of the strip $\RR\times(0,w)$
lying to the right of the line segment $J$ joining the endpoints of the rays $I_1$ and $I_2$.)

The proof of Theorem~\ref{other-negative-theorem}
 is identical to the proof in the case of parallelograms (Theorem~\ref{scherk-gauss-map-theorem}).

\begin{theorem}\label{other-tangency-theorem}
Let $I_1$, $I_2$, $M$, and $u$ be as in Theorem~\ref{other-negative-theorem}.
Let $G$ be a tilted grim reaper defined over an open strip $S$.
\begin{enumerate}[\upshape (1)]
\item\label{other-tangency-one} If $S$ is disjoint from $I_1$, then $G$ and $M$ are nowhere tangent.
\item\label{other-tangency-two}  If $I_2$ is a ray and if $S$ is disjoint from $I_2$, then $G$ and $M$ are nowhere tangent.
\end{enumerate}
\end{theorem}

If $I_2$ is a line (the line $y=w$),
we do not know whether $S$ being 
disjoint from $I_2$ prevents $G$ from being tangent to $M$.

\begin{proof}[Proof of Theorem~\ref{other-tangency-theorem}]
We use the following fact:
\begin{enumerate}[(*)]
\item If $S$ is a nonhorizontal strip that is disjoint from $I_1$ or from $I_2$, 
or, more generally, if $S$ is a limit of such strips, then $G$ and $M$ are nowhere tangent.
\end{enumerate}
The fact was proved as part of the proof of Proposition~\ref{reaper-count-proposition}.

If $S$ is disjoint from the ray $I_1$, then rotating it counterclockwise about the origin through
a small angle $\theta$ produces a nonhorizontal  strip $S_\theta$ that is disjoint from $I_1$.
Thus Assertion~\ref{other-tangency-one} follows from the fact~\thetag{*}.

Assertion~\ref{other-tangency-two} is proved in exactly the same way. (One rotates about the endpoint of $I_2$.)
\end{proof}

In the preceding proof, we cannot handle the case when $I_2$ is a line and  $S$ contains $I_1$,
because in that case,  any nonhorizontal 
strip $S'$ sufficiently close to $S$ will intersect both $I_1$ and $I_2$, and so we cannot conclude anything
from the fact~\thetag{*}.



\section{Morse Theory}\label{morse-section}

Now we explain the Morse theory that was used in the previous section.
Suppose $\Omega\subset\RR^2$ is a connected open set and that
$u,v:\Omega\to\RR$ are translators such that $u-v$ is not constant.
If $p\in \Omega$ is a critical point of $u-v$, then, after a suitable affine change of coordinate that
takes $p$ to the origin, $u-v$ has the form
\begin{equation}\label{the-form}
     c + r^n \sin n\theta + o(r^n).
\end{equation}
for some integer $n\ge 2$.
We shall say that $p$ is a critical point (or saddle point) of $u-v$ with {\bf multiplicity} $n-1$.

The following theorem is about functions whose critical points have the form~\eqref{the-form}.

\begin{theorem}\label{morse-type-theorem}
Suppose $M$ is a compact, smooth surface-with-boundary and that $F:M\to \RR$ is a smooth
function such that
\begin{enumerate}
\item $\nabla(F|M)$ does not vanish at any point of $\partial M$.
\item $F|\partial M$ is a Morse function.
\item At each critical point of $F$, $F$ has the form~\eqref{the-form} in suitable local coordinates.
\end{enumerate}
Define $c_i=c_i(F|M)$ for $i=0,1$ as follows.
\begin{enumerate}
\item $c_0$ is the number of local minima of $F|\partial M$ that are also local minima of $F|M$.
\item $c_1$ is the number of local maxima of $F|\partial M$ that are not local maxima of $F|M$.
\end{enumerate}
Then 
\[
   N(F|M) = c_0 - c_1 - \chi(M), \tag{*}
\]
where $N(F|M)$ is the number of critical points (counting multiplicity) of $F|M$ and 
where $\chi(M)$ is the Euler characteristic of $M$.

More generally, if $M_a = M\cap \{F\le a\}$, then
\begin{equation}\label{general-counting-formula}
  N(F| M_a) = c_0(F|M_a) - c_1(F|M_a) - \chi(M_a)
\end{equation}
Furthermore, for~\eqref{general-counting-formula}
 to hold, it is not necessary for $M$ to be compact; it suffices for $M_a$ to be compact.
\end{theorem}

\begin{proof}
Consider how $\chi(M_t)$ changes as we increase $t$.  When $t< \min_M F$,
$\chi(M_t)$ is zero.   When we pass a local minimum of $F|\partial M$ that is also
a local minimum of $F|M$, then $\chi(M_t)$ increases by $1$.
Similarly, when we pass a local maximum of $F|\partial M$ that is not a local maximum
of $F|M$,  $\chi(M_t)$ decreases by $1$. 
When we pass a critical point of $F$ with multiplicity $k$, then $\chi(M_t)$ decreases
by $k$.

Otherwise, $\chi(M_t)$ does not change.
In particular, passing a local maximum of $F|\partial M$
that is also a local maximum of $F|M$ does not change $\chi(M_t)$, nor does passing
a local minimum of $F|\partial M$ that is not a local minimum of $F|M$.
\end{proof}

\begin{theorem}\label{morse-family-theorem}
Let $M$ be a smooth surface without boundary and $F_\theta: M\to \RR$, $\theta\in [0,T]$, be
a one-parameter family of smooth functions such that the critical points all have the form~\eqref{the-form}. 
Suppose $p\in M\setminus \partial M$ is a critical point of multiplicity $k$ of $F_0$.
Then there is a neighborhood $D\subset M$ of $p$ and an $\eps>0$ such that 
$N(F_\theta|D)=k$ for all $\theta \in [0,\eps]$.
\end{theorem}

\begin{proof}
Let $D\subset M$ be a compact region with smooth boundary such that $p\in D\setminus\partial D$
and such that $D$ contains no other points where $\nabla F_0$ vanishes.
By perturbing $D$ slightly, we can assume that $F_0|\partial D$ is a Morse function.
Let $J$ be the set of $\theta\in I$ such that $\nabla F_\theta$ does not vanish anywhere on $\partial D$
and such that $F_\theta|\partial D$ is a Morse function.  Then $J$ is a relatively open subset of $[0,T]$
containing $0$.  
By Theorem~\ref{morse-type-theorem}, $N(F_\theta|D)$ is constant on each connected component of $J$.
\end{proof}

\begin{corollary}\label{morse-family-corollary}
If $k$ is an integer and $N(F_0|M)\ge k$, then $N(F_\theta|M)\ge k$ for all $\theta$ sufficiently close to $0$.
\end{corollary}






\section{Semi-infinite strips}

In this section, we prove some basic properties of translating graphs defined over
semi-infinite strips.  These properties will be used repeatedly in subsequent sections.
The reader may wish to skip this section and refer back to it when needed.

\begin{theorem}\label{no-translator}
No translator is the graph of a function
\[
   u: (0,\infty)\times (0,w)\to \RR
\]
such that $u(\cdot,0)=u(\cdot,w)= + \infty$.
\end{theorem}

\begin{proof}
Suppose there were such an $u$.   Let 
\[
   g: (\pi,2\pi)\times \RR \to \RR
\]
be a grim reaper surface.   Then $g-u$ would have boundary values $-\infty$ at all
points of the boundary of the square $(\pi,2\pi)\times (0,w)$ and thus would attain
a maximum at an interior point of the square, violating the strong maximum principle.
\end{proof}

Recall that a grim reaper surface is a tilted grim reaper of tilt $0$.

\begin{lemma}\label{half-strip-lemma}
Let $M$ be a translator that is the graph of a function
\[
   u: (0,\infty)\times (0,w)\to\RR
\]
with boundary values $ \infty$ or $-\infty$ on each horizontal edge of the domain.
If 
\[
    \pdf{u}{y}(x_i,y_i) = 0
\]
and $x_i\to \infty$, then, after passing to a subsequence, $y_i$ converges to a limit $\hat{y}$ and
the surfaces 
\[
   M_i = M-(x_i,0, u(x_i,y_i))
\]
converge smoothly to a complete translator.  The component containing the point $(0,\hat{y},0)$
 is a tilted grim reaper or a $\Delta$-wing
defined over the strip  $\RR\times I$ for some interval $I$ in $(0,w)$. Furthermore,
\begin{equation}\label{away-from-edges}
   \frac{\pi}2 \le \hat{y} \le w - \frac{\pi}2,
\end{equation}
\begin{equation}\label{bends-down}
     \lim_{i\to\infty}\left(\frac{\partial}{\partial y}\right)^2 u(x_i,y_i) < 0,
\end{equation}
and
\begin{equation}\label{trivial}
  \lim_{i\to\infty} |Du(x_i,y_i)| <\infty.
\end{equation}
\end{lemma}

In fact, as we will see later (Theorem~\ref{half-reaper}), $I=(0,w)$ and the limit surface must
be a tilted grim reaper, not a $\Delta$-wing.

\begin{proof}
Smooth subsequential convergence of the $M_i$ 
follows from the standard curvature estimates (e.g., Remark~\ref{dichotomy}). 
If the component $M'$ of the limit translator containing the point $(0,\hat{y},0)$ were not a graph,
it would be a vertical plane in the slab $\{0\le y\le w\}$, and hence
the plane $\{y=\hat{y}\}$.  But that is impossible since $\ee_2=(0,1,0)$
is tangent to $M'$ at $(0,\hat{y},0)$.  Thus $M'$ is the graph of a function $u'$ over a domain
contained in $\RR\times (0,w)$.  By the Classification Theorem~\ref{classification-theorem},
  $M'$ is a tilted grim reaper surface or a $\Delta$-wing, and the domain is a strip $\RR\times I$
  of width $\ge \pi$.
Since $\pdf{u'}y(0,\hat{y})=0$, the point $(0,\hat{y})$ is in the center line of the strip $\RR\times I$,
and thus~\eqref{away-from-edges} holds.

Note that~\eqref{bends-down} holds since if $g:\RR\times I\to \RR$ is a tilted grim reaper or a $\Delta$-wing, 
 then $(\partial/\partial y)^2g$
is strictly negative everywhere.

Finally,~\eqref{trivial} follows trivially from the smooth convergence to a graphical limit.
\end{proof}

\begin{theorem}\label{vrock}
Suppose $M$ is a translator that is the graph of a function
\[
   u: (0,\infty)\times (0,w)\to \RR
\]
with boundary values $u(\cdot,0)\equiv -\infty$ and $u(\cdot,w)\equiv +\infty$
on the horizontal edges.   
 If $p_i=(x_i,y_i,z_i)\in M$ and $x_i\to \infty$, then $M-p_i$ converges smoothly
to the plane $y=0$, and the upward unit normal to $M-p_i$ converges to $-\ee_2$.
\end{theorem}

\begin{proof}
First, we claim that there is an $a$ such that 
\begin{equation}\label{up}
  \pdf{u}y>0  \quad\text{on $(a,\infty)\times (0,w)$}.
\end{equation}
If not, there would be a sequence $(x_i,y_i)$ with $x_i\to\infty$ and with
\begin{equation}\label{down}
  \pdf{u}y(x_i,y_i)\le 0.
\end{equation}
Since $u(x_i,y)\to\infty$ as $y\to w$, we can choose $y_i$ to be the largest value in $(0,w)$
for which~\eqref{down} holds.
Then 
\[
   \pdf{u}{y}(x_i,y_i)=0 
\]
and
\[
  \pdf{u}{y}(x_i,y)>0 \quad\text{for $y\in (y_i,w)$},
\]
so $(\partial/\partial y)^2u(x_i,y_i)\ge 0$, contradicting~\eqref{bends-down}.
Thus we have proved~\eqref{up}.

Now let $(x_i,y_i)$ be any sequence with $x_i\to\infty$.
By passing to a subsequence, we can assume that the upward unit normal $\nu_i$
to $M$ at $p_i=(x_i,y_i,u(x_i,y_i))$ converges to a unit vector $\nu$.

Then (after passing to a subsequence) the surfaces $M - p_i$ converge smoothly
to a complete translator.  Let $M'$ be the component containing the origin.
If $\nu\cdot\ee_3>0$, then $M'$ would be the graph of a function $u':\RR\times I\to \RR$.
By~\eqref{up},
\[
   \pdf{u'}y\ge 0
\]
everywhere.  But there is no such complete translator $u'$.

Thus $\ee_3\cdot\nu=0$, so $M'$ must be a vertical plane.  Since $M'$ is contained in a slab
of the from $\{(x,y,z): y\in I\}$, in fact $M'$ must be the plane $\{y=0\}$. 
Thus $\nu=\pm \ee_2$.  Since $\nu_i\cdot\ee_2<0$
for all sufficiently large $i$ (by~\eqref{up}), we see that $\nu=-\ee_2$.
\end{proof}

\begin{theorem}\label{half-reaper}
Suppose $M$ is a translator that is the graph of a function
\[
  u: (0,\infty)\times (0, w) \to \RR
\]
with boundary value $-\infty$ on the horizontal edges:
\[
   u(\cdot,0)=u(\cdot,w) =  - \infty.
\]
Then $w\ge \pi$, and $M$ weakly asymptotic to a tilted grim reaper as $x\to\infty$
in the following sense.  There is a tilted grim reaper
\[
  g: \RR\times (0,w) \to \RR
\]
such that
\[
   M - (x, 0, u(x,0))
\]
converges smoothly to the graph of $g$ as $x\to\infty$.
\end{theorem}

\begin{proof}
For each $x$, the function $y\mapsto u(x,y)$ attains its maximum since it tends to  $-\infty$
at the boundary points.  Thus it has at least one critical point.
By~\eqref{bends-down} in Lemma~\ref{half-strip-lemma},
 there is an $A$ such that for $x\ge A$, every critical point of $y\mapsto u(x,y)$
is a strict local maximum.   Thus there is a unique critical point $y(x)$ (since between two
strict local maxima there is a local minimum), and $y(x)$ depends smoothly on $x$.

By Lemma~\ref{half-strip-lemma}, 
\begin{equation}\label{away}
   \frac{\pi}4 < y(x) < w - \frac{\pi}4
\end{equation}
for all sufficiently large $x$.
By replacing $A$ by a larger number, we can assume that~\eqref{away} holds for all $x\ge A$.

Also, $|Du(x,y(x))|$ is bounded above by~\eqref{trivial}.

Thus we have shown: $|Du|$ is bounded above on a unbounded, connected  set
(namely the curve $\{(x,y(x)): x\ge A\}$) contained in $[A,\infty]\times [\pi/4, w- \pi/4]$.
By~\cite{himw}*{Lemma~6.3},
\[
   \sup_{[A,\infty)\times K} |Du| < \infty
\]
 for each compact subinterval $K$ of $(0,w)$.
Consequently, every sequence $x_i\to\infty$
has a subsequence for which
\[
   M - (x_i,0,u(x_i,w/2))
\]
converges to a complete translator that is the graph of a function
\[
    g: \RR\times (0,w)\to\RR.
\]
with $g(0,0)=0$.
Note that the family $\Ff$ of such subsequential limit functions is connected.
By the Classification Theorem~\ref{classification-theorem}, each $g\in \Ff$ is a tilted grim reaper
or a $\Delta$-wing.

To prove the theorem, it suffices to show that $\Ff$ contains only one
element, a tilted grim reaper.

{\bf Case 1}: The width $w=\pi$.  
The only complete translator over a strip of width $\pi$ is the grim reaper surface,
so the theorem is true in this case.

{\bf Case 2}: The width $w$ is $>\pi$.
If $\Ff$ contained a $\Delta$-wing, then the function $x\mapsto u(x,w/2)$
would have an unbounded sequence of strict local maxima.  Since there is a local
minimum between each pair of local maxima, there would be an unbounded
sequence $x_i$ of local minima of $u(\cdot,w/2)$.
Thus (after passing to a further subsequence)
\[
   M - (x_i,0,u(x_i,0))
\]
would converge to a tilted grim reaper or $\Delta$-wing $g\in \Ff$ with $(\partial /\partial x)g(0,0)=0$ and 
$(\partial/\partial x)^2 g \ge 0$.
The only such $g$ is the grim reaper surface, which is not in $\Ff$ since the width is $>\pi$.

Thus $\Ff$ contains only tilted grim reapers.
There are only two tilted grim reapers over $\RR\times(0,w)$ that contain the point $(0,w/2,0)$.
Thus since $\Ff$ is connected, it contains only one of those two surfaces.
\end{proof}


  


 

\section{Scherkenoids} \label{scherkenoid-section}


For $\alpha\in (0,\pi)$ and $w\in (0,\infty)$, let
\[
  \Omega(\alpha,w)= \{ 0<y<w\} \cap \{ x> y/\tan\alpha\}.
\]
In the notation of Definition~\ref{P-notation}, $\Omega(\alpha,w)=P(\alpha,w,\infty)$.

\begin{definition}\label{g_w-definition}
For $w\ge \pi$, let $g_w:\RR\times (0,w)$ be the unique tilted grim reaper function
such that $g_w(0,w/2)=0$ and such that $\partial g_w/\partial x \le 0$.
Let $g_w'$ be the corresponding tilted grim reaper function with $\partial g_w'/\partial x\ge 0$.
(Thus $g_w'(x,y)\equiv g_w(-x,y)$.)
\end{definition}

Consequently 
\begin{equation*}
  g_w(x,y) = (w/\pi)^2 \log (\sin (y (\pi/w)) - x \sqrt{(w/\pi)^2 - 1}
\end{equation*}
(see~\eqref{TGR2}), and therefore
\[
  \pdf{g_w}x \equiv -\sqrt{ (w/\pi)^2-1}.
\]

\begin{theorem}\label{scherkenoid-theorem}
For every $\alpha\in (0,\pi)$ and $w\in [\pi,\infty)$, there exists a smooth
translator $\Dd$ such that $M\setminus \partial M$ is the graph of a function
\[
  u: \Omega(\alpha,w)\to \RR
\]
with boundary values $-\infty$ on the horizontal edges of $\Omega(\alpha,w)$
    and $+\infty$ on the non-horizontal edge. 
Furthermore, any such $\Dd$ has the following properties:
\begin{enumerate}[\upshape (1)]
\item\label{noid-curvature-item}
 $\Dd$ has negative Gauss curvature everywhere.
\item\label{noid-reaper-item}
As $a\to\infty$, 
\[
    u(a+x,y) - u(a,y)
\]
converges smoothly to the tilted grim reaper function $g_w(x,y)$. 
\item\label{noid-gauss-image-item}
The Gauss map takes $\Dd\setminus \partial \Dd$  diffeomorphically onto the
open region $\Rr=\Rr(w)$ in the upper hemisphere bounded by $C\cup C(w)$, 
where $C$ is the equatorial semicircle
\[
  C = \{(x,y,0)\in \SS^2: x \ge 0\}
\]
and where $C(w)$ is the great semicircle that
is the image of $\operatorname{graph}(g_w)$ under its Gauss map:
\[
    C(w) =  \left\{ (x,y,z)\in \SS^2: \text{$z>0$ and $x = z\,\sqrt{(w/\pi)^2-1}$} \right\}.
\]
\item\label{noid-unique-item} $\Dd$ is unique up to vertical translation.
In particular, there is a unique such surface $\Dd_{\alpha,w}$ with the additional property that
the vector
\[
    \vv_{\alpha/2}=(\cos(\alpha/2),\sin(\alpha/2),0)
\]
is tangent to $\Dd$ at the origin.
\end{enumerate}
\end{theorem}

\begin{figure}[htbp]
\begin{center}
\includegraphics[width=.46\textwidth]{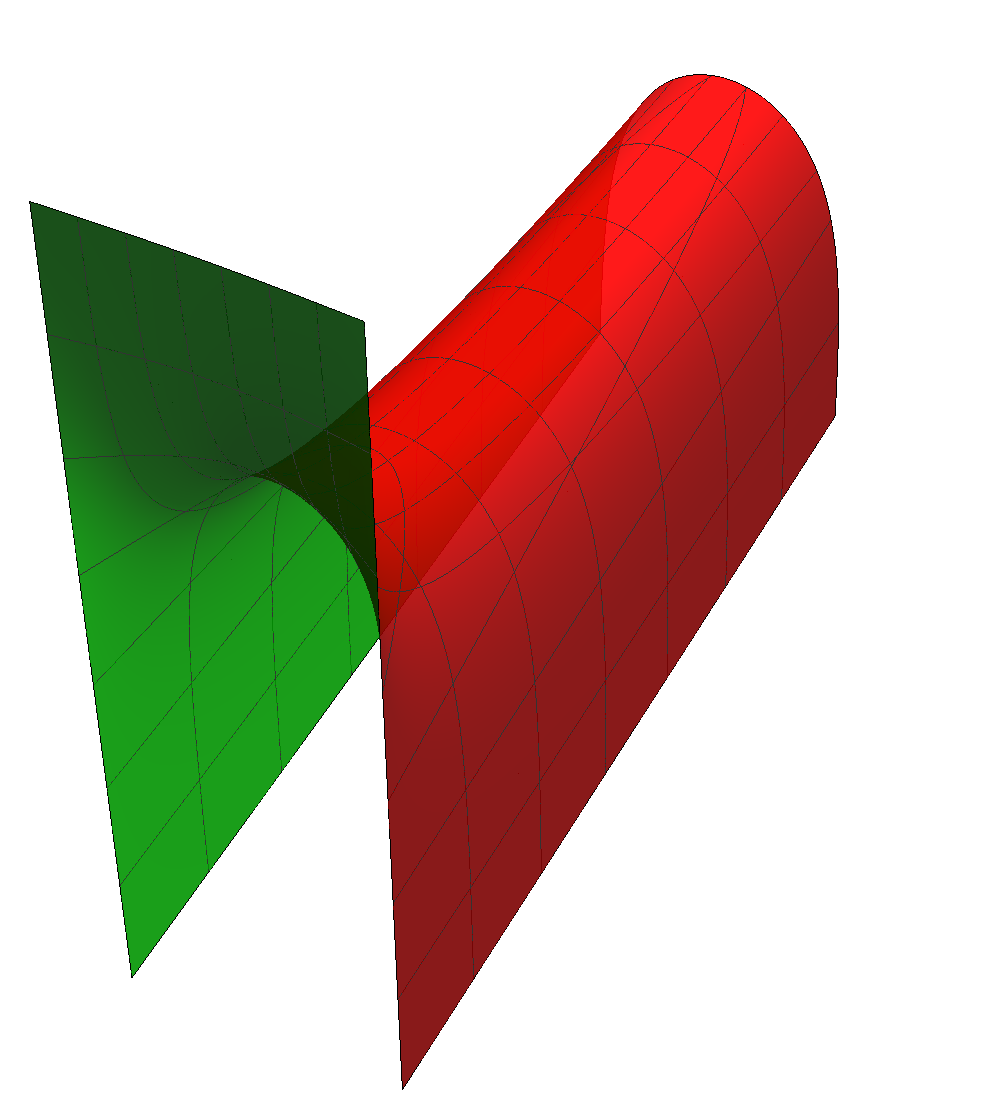}
\caption{The graph of the scherkenoid function $u$ in Theorem~\ref{scherkenoid-theorem}
with $\alpha=\pi/2$ and $w= \pi$.}
\label{default}
\end{center}
\end{figure}

\begin{remark} If $w<\pi$, there is no translator $u$ with the indicated boundary values.
 This follows from Theorem~\ref{half-reaper}.
\end{remark}

\begin{corollary}\label{scherkenoid-corollary}
If $\Dd$ and $u$ are as in Theorem~\ref{scherkenoid-theorem},
\begin{enumerate}[\upshape(1)]
\item The total curvature (i.e., the integral of the absolute value of the Gauss curvature)
of $\Dd$ is 
$
  \area(\Rr(w)) = 2\arcsin(\pi/w).
$
\item
The Gauss map $\nu$ is a diffeomorphism from $M$ onto
$
   \Rr \cup C \setminus \{\ee_2,-\ee_2,\hat{n}\},
$
where $\hat{n}=(-\sin\alpha, \cos\alpha,0)$.
\item $\partial u / \partial x < -\sqrt{(w/\pi)^2-1}$.
\end{enumerate}
\end{corollary}

Recall that $\Ss_{\alpha,w}$ is
the complete, singly periodic surface 
obtained from $\Dd_{\alpha,w}$ by iterated Schwarz reflection.
Note that $\Ss_{\alpha,w}$ is simply connected (by Van Kampen's Theorem, for example).
Hence by Corollary~\ref{scherkenoid-corollary}, 
the Gauss map on $\Ss_{\alpha,w}$
may be regarded as a universal cover of $\Rr\cup\Rr'\cup C\setminus \{\hat{n}\}$,
where $\Rr'$ is the image of $\Rr$ under the reflection $(x,y,z)\mapsto (x,y,-z)$.

The proof of Theorem~\ref{scherkenoid-theorem} is somewhat long,
so we divide it in three parts: proof of existence, 
proof of Assertions~\ref{noid-curvature-item}--\ref{noid-gauss-image-item},
and proof of uniqueness (Assertion~\ref{noid-unique-item}).

\begin{proof}[Proof of Existence in Theorem~\ref{scherkenoid-theorem}]
For $c>0$, let 
\[
  \Omega_{\alpha,w}(c) = \{(x,y)\in \Omega(\alpha,w): x< c\}.
\]
We work with $c$ sufficiently large that $\Omega_{\alpha,w}(c)$ is a quadrilateral.
Let $h>0$.
Consider the translator $u^h_c:\Omega_{\alpha,w}(c)\to \RR$ 
such that $u^h_c=h$ on the left edge of $\Omega_{\alpha,w}(c)$ and such that $u^h_c=0$ 
on the other three edges.    (Of course $u^h_c$ also depends on $\alpha$ and $w$,
but for the moment we fix $\alpha$ and $w$.)

As $h\to\infty$, the graph of $u^h_c$ converges (perhaps after passing to a subsequence)
 to a limit
surface $\Sigma=\Sigma_c$ whose boundary is a polygonal curve $\Gamma$ consisting of the
two vertical rays with $z>0$ over $(0,0)$ and $\left(\frac{w}{\tan \alpha},w\right)$
together with three
sides of the quadrilateral $\Omega_{\alpha,w}(c)$ in the plane $z=0$.  The convergence is smooth
except possibly at the corners of $\Gamma$. (See Remark~\ref{dichotomy}).   If $\Sigma$ were not a graph, 
it would be flat and vertical.  But $\Gamma$ does not bound such a flat, vertical surface.
Thus $\Sigma$ is the graph of a function $u_c$.  The proof of Lemma~\ref{straight-lemma}
shows that the domain of $u$ is a convex open polygonal region in the quadrilateral $\Omega_{\alpha,w}(c)$
and that the portion of the boundary of the domain where $u_c=\infty$ is the straight line
segment joining $(0,0)$ and $\left(\frac{w}{\tan \alpha},w\right)$.   
Thus the domain of $u_c$ is all of $\Omega_{\alpha,w}(c)$, and $u_c$ has boundary value $+\infty$
on the left side of the quadrilateral $\Omega_{\alpha,w}(c)$ and $0$ on the other three sides.

Let 
\[
   v_c: [-c,c]\times [0,w] \to \RR
\]
be the translator with boundary values $0$.
By the maximum principle, $v_c\le u_c$ on the intersection of the two domains.
In the proof of~\cite{himw}*{Theorem~4.1},
it is shown that $v_c-v_c(0,0)$ converges (as $c\to\infty$) to a complete
translating graph over $\RR\times(0,w)$.  Thus $u_c(x,y)\to\infty$ for every $(x,y)\in \Omega(\alpha,w)$.
Consequently, $\Sigma_c$ converges (as $c\to\infty$) to the pair
of vertical quarterplanes 
\begin{equation}\label{thing-one}
   \{(x,0,z): x\ge 0, \, z\ge 0\}
\end{equation}
and
\begin{equation}\label{thing-two}
    \{ (x,w,z):  x\ge \hat{x}, \,z\ge 0\}
\end{equation}
where $\left(\hat{x}=\frac{w}{\tan \alpha},w \right)$ is the upper-left corner of $\Omega(\alpha,w)$.
By the standard curvature estimates (see Remark~\ref{dichotomy}), the convergence is smooth.

Let $(\cos\theta_c(z),\sin\theta_c(z),0)$ be tangent to the graph of $u_c$ at $(0,0,z)$.
Then $\theta_c(0)=0$ and $\theta_c(z)\to\alpha$ as $z\to\infty$.
Thus there is a $z(c)$ such that $\theta_c(z(c))=\alpha/2$.

By the smooth convergence of $\Sigma_c$ to the quarterplanes~\eqref{thing-one} and~\eqref{thing-two},
 we see that 
\[
  \lim_{c\to\infty}z(c) = \infty.
\]

Now take a subsequential limit $\Dd$ of the graph of $u_c-z(c)$ as $c\to\infty$.
Then $\Dd$ is a smooth translator in $\overline{\Omega(\alpha,w)}\times \RR$ whose boundary
consists of the two vertical lines through the corners of $\Omega(\alpha,w)$,
and $\vv:=(\cos\alpha/2,\sin\alpha/2,0)$ is tangent to  $\Dd$ at the origin.

Let $\Dd_0$ be the component of $\Dd$ containing the origin.
If $\Dd_0\setminus\partial \Dd_0$ were not a graph, then $\Dd_0$ would be flat and vertical, and
hence the halfplane $\{s\vv+ z\ee_3: s\ge 0, z\in \RR\}$.  But this is impossible, since that
halfplane is not contained in $\overline{\Omega(\alpha,w)}\times \RR$.
Thus $\Dd_0\setminus \partial \Dd_0$ is the graph of a function $u$.  
By Lemma~\ref{straight-lemma}, the domain of $u$ must be all of $\Omega$, and
 therefore $\Dd_0$ must be all of $\Dd$.
This completes the proof of existence in Theorem~\ref{scherkenoid-theorem}.
\end{proof}

\begin{proof}[Proof of Assertions~\ref{noid-curvature-item}--\ref{noid-gauss-image-item}
                in Theorem~\ref{scherkenoid-theorem}]
That $\Dd$ has negative Gauss curvature everywhere
 (Assertion~\ref{noid-curvature-item}) was part of Theorem~\ref{other-negative-theorem}. 
 
To prove Assertion~\ref{noid-reaper-item},
 note by Theorem~\ref{half-reaper} that $u(a+x,y)-u(a,w/2)$ converges smoothly as $a\to\infty$
either to $g_w$ 
or to $g_w'$ (as in Definition~\ref{g_w-definition}.)
If $w=\pi$, then $g'_w=g_w$, so we are done.  Thus we may assume that $w>\pi$.
Consider a grim reaper function $g$ over the strip
\[
    \RR\times (w/2-\pi/2, w/2+\pi/2).
\]
If $u(a+x,x)-u(a,w/2)$ converged to $g'_w$, then 
\[
   g-u: \RR\times (-w/2,w/2) \to \RR
\]
 would attain its maximum, violating the strong maximum principle.
(To see that $g-u$ attains its maximum, note that the assumption $u\sim g_w'$ implies that
$(g-u)(p_i)\to -\infty$ for any divergent sequence $p_i$ in $\RR\times (-w/2,w/2)$.)
This completes the proof of Assertion~\ref{noid-reaper-item}.

To prove Assertion~\ref{noid-gauss-image-item},
consider the set $\Gamma$ of all $v \in \SS^2$ for which there is a
 divergent sequence $p_i\in \Dd\setminus \partial \Dd$
such that $\nu(p_i)\to v$.  
By Assertion~\ref{noid-reaper-item} and by the boundary values of $u$, we see
that $\Gamma$ is the simple closed curve $C\cup C(w)$.  
Since 
\[
         \nu: \Dd\setminus \partial \Dd   \to   \SS^2
\]
is a smooth immersion (by Assertion~\ref{noid-curvature-item}), 
it follows that the multiplicity function
\[
     v\in \SS^2 \mapsto \#\{p\in M: \nu(p)=v\}
\]
is constant on each of the two connected components of $\SS^2\setminus \Gamma$.
The function is $0$ on the lower hemisphere, so it is $0$ on all of $\SS^2\setminus \overline{\Rr}$.
Thus 
\[
  \nu:\Dd\setminus\partial \Dd\to \Rr \tag{*}
\]
 is a proper immersion, i.e., a covering map.
Since $\Dd \setminus \partial \Dd$ is connected and $\Rr$ is simply connected,
the map~\thetag{*} must be a diffeomorphism.
\end{proof}

To complete the proof of Theorem~\ref{scherkenoid-theorem}, it remains only
to prove uniqueness (Assertion~\ref{noid-unique-item}):

\begin{theorem}[Scherkenoid Uniqueness Theorem]\label{scherkenoid-uniqueness}
Suppose 
\[
  u,v: \Omega=\Omega(\alpha,w)\to \RR
\]
are Scherkenoid functions (i.e., translators with boundary values $-\infty$ on the boundary rays
and $+\infty$ on the boundary segment.)
Then $u-v$ is constant.
\end{theorem}

\begin{lemma}\label{noid-unique-lemma}
Let $\Omega=\Omega(\alpha,w)$ and $\Omega'=\Omega+\xi$, where $\xi=(\xi_1,\xi_2)$ is a nonzero vector.
Let $u$ and $u'$ be Scherkenoid functions on $\Omega$ and $\Omega'$.
Then $u'-u$ has no critical points.
\end{lemma}

\begin{proof}[Proof of Lemma~\ref{noid-unique-lemma}]
By interchanging $\Omega$ and $\Omega'$ (if necessary), we may assume that $\xi_2>0$
or that $\xi_1 > 0 = \xi_2$.
Suppose, contrary to the lemma, that $u'-u$ has a critical point $p_0$.
Note that $p_0$ is an isolated critical point, since otherwise $u'-u$ would be
be constant in a neighborhood of $p_0$ and therefore (by unique continuation)
constant on all of $\Omega\cap\Omega'$, which is impossible (since $u'-u$
has boundary values $+\infty$ on some boundary components and $-\infty$ on other
boundary components.)

Rotate $u': \Omega'\to\RR$ counterclockwise by $\theta$ about the 
 lower-left corner of $\Omega$ to get 
\[
   u'_\theta: \Omega'_\theta \to \RR.
\]
Note that for all sufficiently small $\theta>0$, 
the domain $\Omega\cap \Omega'_\theta$
is a triangle.  
For such a $\theta$, the function $u'_\theta-u_\theta$ has no critical points
by Proposition~\ref{rado-proposition}.  
(In our situation, the set $S$ in that proposition contains exactly two points: one point is a vertex
of one region that lies in the interior of the other region, and the other point is a point where the lower boundary
ray of $\Omega'_\theta$ intersects the upper boundary ray of $\Omega$.)

But by Corollary~\ref{morse-family-corollary},
 for all $\theta>0$ sufficiently small, $u'_\theta-u$ must have one or more critical points near $p$,
a contradiction.  This proves the lemma.
\end{proof}

\begin{proof}[Proof of Theorem~\ref{scherkenoid-uniqueness}]
Assume to the contrary that $u-v$ is not constant.   Then there is a point $p_0$ in  $\Omega=\Omega(\alpha,w)$
such that $Du(p_0)\ne Dv(p_0)$.
By Theorem~\ref{scherkenoid-theorem}\eqref{noid-gauss-image-item}, 
 the graphs of $u$ and of $v$ have the same Gauss image.
Hence there is a point $q_0$ in $\Omega(\alpha,w)$ such that $Du(q_0)=Dv(p_0)$.  Consequently, 
if we translate $u: \Omega\to\RR$ by $p_0-q_0$ to get $u':\Omega'\to \RR$, then
$q_0$ is a critical point of $u'-v$.  But that contradicts Lemma~\ref{noid-unique-lemma}, thus proving the theorem.
\end{proof}


\begin{theorem}
Suppose that  $(\alpha_i,w_i)$ converges to $(\alpha,w)$ in $(0,\pi)\times (0,\infty)$.
Then
$\Dd_{\alpha_i,w_i}$ converges smoothly to $\Dd_{\alpha,w}$, and
$\Ss_{\alpha,w}$ converges smoothly to $\Ss_{\alpha,w}$.
\end{theorem}

\begin{proof}
For $w<\pi$, this was already proved (Theorem~\ref{scherk-convergence-theorem}
and Corollary~\ref{scherk-convergence-corollary}).  Thus consider the case $w\ge \pi$.
By passing to a subsequence, the $\Dd_i=\Dd_{\alpha_i,w_i}$ converge smoothly
to a limit surface $\Dd$ in $\overline{\Omega(\alpha,w)}\times \RR$.
By the smooth convergence, the vector $\vv=(\cos(\alpha/2),\sin(\alpha/2),0)$ is tangent to $\Dd$ at
the origin.  If $\Dd\setminus \partial \Dd$ were not a graph, then $\Dd$ would be the vertical
halfplane $\{ s\vv+t\ee_3: s\ge 0, t\in\RR\}$, which is impossible since that halfplane is not
contained in $\overline{\Omega(\alpha,w)}\times \RR$.  Thus $\Dd\setminus\partial\Dd$ is the 
graph of a function 
\[
   u:\Omega' \subset \Omega(\alpha,w)\to \RR.
\]
By Lemma~\ref{straight-lemma}, $\Omega'$ is a convex open set and each component
of
\[
   (\partial \Omega')\setminus \{\textnormal{the corners of $\Omega(\alpha,w)$} \}
\]
is a ray, segment or line.  It follows immediately that the $\Omega'$ is all of $\Omega(\alpha,w)$.
The smooth convergence implies that $u$ has the correct boundary values.
Thus $\Dd=\Dd_{\alpha,w}$, so we have proved that $\Dd_{\alpha_i,w_i}$ converges smoothly
to $\Dd_{\alpha,w}$.
It follows immediately that $\Ss_{\alpha_i,w_i}$ converges smoothly to $\Ss_{\alpha,w}$.
\end{proof}


\section{Letting $\alpha$ tend to $\pi$}

For $0< \alpha < \pi$ and for $0<w< \infty$,
let $\Dd_{\alpha,w}$ be the fundamental piece of the Scherk translator
(if $w<\pi$) or of the Scherkenoid (if $w\ge \pi$)
as in Theorems~\ref{scherk-translator-theorem} and~\ref{scherkenoid-theorem}.
In particular, $\Dd_{\alpha,w}\setminus \partial \Dd_{\alpha,w}$ is the graph of a function $u_{\alpha,w}$ over
the parallelogram $P(\alpha,w,L(\alpha,w))$ if $w<\pi$, and a graph over the region
$P(\alpha,w,\infty)$ if $w>\pi$.
The lower-left corner of the domain is at the origin, and the vector 
\newcommand{\uu}{\mathbf{u}}
\[
   \vv_{\alpha/2} = (\cos(\alpha/2), \sin(\alpha/2),0)
\]
is tangent to $\Dd_{\alpha,w}$ at the origin.

\begin{theorem}\label{pi-theorem}
Suppose that $\alpha_i\in (0,\pi)$ converges to $\pi$ and that $w_i\in (0,\infty)$ converges to $w\in (0,\infty)$.
Then, after passing to a subsequence, $\Dd_{\alpha_i,w_i}$ converges smoothly
to a translator $M$ such that $M\setminus \partial M$ is the graph of a function
\[
   u: \RR\times (0,w) \to \RR.
\]
\begin{enumerate}[\upshape (1)]
\item If $w< \pi$, then $\partial M$ consists of two vertical lines, the $z$-axis and the line
through a point $(\hat{x},w)$ on the line $y=w$.  In this case (see Figure \ref{default-3}), $u$ has boundary values
\[
   u(x,0) =
   \begin{cases} 
   +\infty &\text{if $x<0$} \\
   -\infty &\text{if $x>0$}
   \end{cases}
\quad
,
\quad
   u(x,w) =
   \begin{cases} 
   -\infty &\text{if $x<\hat{x}$}, \\
   +\infty &\text{if $x>\hat{x}$}.
   \end{cases}
\]
Repeated Schwarz reflection produces a complete, simply connected, 
singly periodic translator~$\widehat{M}$.
\item
If $w\ge\pi$, then the boundary of $M$ is a single vertical line, the $z$-axis $Z$.
In this case (see Figure \ref{default-2}), $u$ has boundary values 
\[
   u(x,0) =
   \begin{cases} 
   +\infty &\text{if $x<0$} \\
   -\infty &\text{if $x>0$}
   \end{cases}
\quad
,
\quad
   u(\cdot,w) = -\infty.
\]
A single Schwarz reflection about the $z$-axis produces a complete, simply connected
 translator $\widehat{M}$ without boundary.
\end{enumerate}
Furthermore, if $w\ne \pi$, then $\Ss_{\alpha_i,w_i}$ converges smoothly to $\widehat{M}$.
If $w=\pi$, then (after passing to a further subsequence) $\Ss_{\alpha_i,w_i}$ converges to
a smooth limit surface $\Ss'$, and the connected component of $\Ss'$ containing the origin is $\widehat{M}$.
\end{theorem}

\begin{proof}  
By the standard curvature estimates (see Remark~\ref{dichotomy}), 
we can assume (after passing to a subsequence)
smooth convergence to a limit translator $\Dd$. 
Note that $\ee_2$ is tangent to $M$ at the origin.
Let $M$ be the connected component of $\Dd$ containing the origin.

\begin{claim}\label{pi-claim} Suppose $\partial M=Z$.  Then $M\setminus \partial M$ is the graph of function
\[
   u: \RR\times (0,w')\to\RR
\]
for some $w'$ with $\pi \le w' \le w$, where $u(x,0)$ is $+\infty$ for $x<0$ and $-\infty$ for $x>0$
and where $u(\cdot,w')= - \infty$.  
\end{claim}

\begin{proof} Note that  $M\setminus \partial M$ is the graph of a function $u$,
as otherwise (by Lemma~\ref{straight-lemma}) it would be the vertical halfplane $\{x=0,\, y\ge 0\}$,
which is impossible since $M$ is contained in the slab $0\le y\le w$.
The domain $\Omega$ of $u$ is an convex open set containing $(0,0)$ in its boundary, 
and each component of $(\partial \Omega)\setminus \{(0,0)\}$ is a straight line or ray 
(by Lemma~\ref{straight-lemma}.)   
Thus $\Omega=\RR\times (0,w')$ for some $w'\le w$.  By Theorem~\ref{pitchfork-theorem}, $w'\ge \pi$.
This completes the proof of the claim.
\end{proof}

By passing to a subsequence, we can assume that either $w_i\ge \pi$ for all $i$
or that $w_i<\pi$ for all $i$.



{\bf Case 1}: $w_i\ge \pi$ for all $i$.
Then $\partial \Dd_i$ consists of $Z$ and one other vertical line,
and since $\alpha_i\to \pi$, that other vertical line drifts off to infinity as $i\to\infty$.
  Thus $M$
is as described in Claim~\ref{pi-claim}.  
The images under the Gauss map of $\Dd_i\setminus \partial \Dd_i$ and of $M\setminus\partial M$ 
are $\Rr(w_i)$ and $\Rr(w')$ 
by Theorems~\ref{scherkenoid-theorem} and~\ref{pitchfork-theorem}, so
\begin{equation}
    \Rr(w') \subset \lim_i \overline{\Rr(w_i)} = \overline{\Rr(w)},
\end{equation}
which implies that $w'\ge w$.  But $w'\le w$, so $w'=w$, and the theorem is proved in Case 1.

{\bf Case 2}: $w_i< \pi$ for all $i$.

Let $L_i=L(\alpha_i,w_i)$.
Note that upper-right corner of the domain parallelogram is at $(x_i,w)$, where
$x_i:=(L_i - w_i/\tan\alpha_i)$.   Since $\sin\alpha_i<\tan\alpha_i$,
\begin{align*}
  x_i \ge L_i - w/\sin\alpha_i 
  &>0
\end{align*}
by Theorem~\ref{jenkins-serrin-like-theorem}.   Since $\sin\alpha_i\to 0$, we see
that  $L_i\to\infty$, so the upper-left and lower-right corners of $P_i$ drift off to infinity.

We may assume that $x_i$ converges to a limit $\hat{x}\in [0,\infty]$.

If the boundary of $M$ consists of only one line (the $z$-axis), then we are done by Claim~\ref{pi-claim}.

Thus suppose $\partial M$ has two lines, the $z$-axis and the vertical line through $(\hat{x},w)$.
The geodesic in $M$ connecting $(0,0,0)$ and $(\hat{x},w,0)$ is the limit of the 
geodesic $C_i$ in $\Dd_i$ connecting $(0,0,0)$ and $(x_i,w_i,0)$.   The geodesic $C_i$
is unique since $\Dd_i$ is negatively curved, and by symmetry, its midpoint is $(x_i/2,w_i/2,u(x_i/2,w_i/2))$,
where the tangent plane is horizontal.  Thus the tangent plane to $M$ at the
midpoint $(\hat{x}/2,w/2,\hat{z})$ of $C$ is horizontal,
 so $M\setminus \partial M$ is the graph of a function $u$.  
Using Lemma~\ref{straight-lemma}, one concludes that the domain of $u$ is all of $\RR\times (0,w)$,
and (by smooth convergence) that $u$ has the specified boundary values.
By Theorem~\ref{helicoid-theorem}\eqref{helicoid-w-item}, $w<\pi$.

The simple connectivity of $\widehat{M}$ (for any value of $w$)
follows 
from the simple connectivity of $M$ by Van Kampen's Theorem, for example.

It remains only to prove the statements about $\Ss_{\alpha_i,w_i}$.  
After passing to a further subsequence, $\Ss_{\alpha_i,w_i}$ converges to a smooth limit $\Ss$,
and $\widehat{M}$ is the component of $\Ss$ containing the origin.
If $w\ne \pi$, then $\widehat{M}$ is all of $\Ss$.  
For $w>\pi$, this follows from 
Lemma~\ref{down-lemma} below applied to the $\Ss_{\alpha_i,w_i}$.
(From Figure~\ref{fig-new}, we see that any component of $\Ss\setminus \widehat{M}$ 
would have to lie in a slab of the form
$2nw \le y\le (2n+1)w$ or $2nw<-y<(2n+1)w$ for some positive integer $n$;
Lemma~\ref{down-lemma} implies that there are no such components.)

For $w < \pi$, note that $\widehat{M}$ is a graph over $\cup_{n\in\ZZ}\RR\times (nw,(n+1)w)$,
so any component of $\Ss\setminus\widehat{M}$ would have to lie in $\cup_{n\in\ZZ}\{y=nw\}$
and thus would be a vertical plane $y=nw$.  But $\widehat{M}$ contains a vertical line in that plane,
so $\Ss\setminus \widehat{M}$ has no components.
\end{proof}

In the following lemma, one should think of $c$ as some fixed number and of $\alpha$ as being close to $\pi$,
so that $\hat{x}:=w/\tan\alpha<<c$ and therefore $(c-2\hat{x}-1)>>0$.

\begin{lemma}\label{down-lemma}
Suppose that $0<\alpha<\pi$ and that $w\ge \pi$.
Let
\[
  \hat{x} = \frac{w}{\tan\alpha},
\]
(so that $(\hat{x},w)$ is a corner of $P(\alpha,w,\infty)$), and 
if $Q$ is a subset of $\RR^2$, let 
\[
    F(Q):= \sup \{  z:   (x,y,z)\in \Ss(\alpha,w), \, (x,y)\in Q \}.
\]
Suppose that 
\[
    2\hat{x} +1 \le c  \le 0.
\]
Then 
\[
   F( [c,\infty)\times [w,\infty)) \le  F( \{1\}\times (0,w))  -  (c-2\hat{x}-1) \sqrt{ (w/\pi)^2 - 1 }.
\]
\end{lemma}
 \begin{figure}[htbp]
\begin{center}
\includegraphics[width=.70\textwidth]{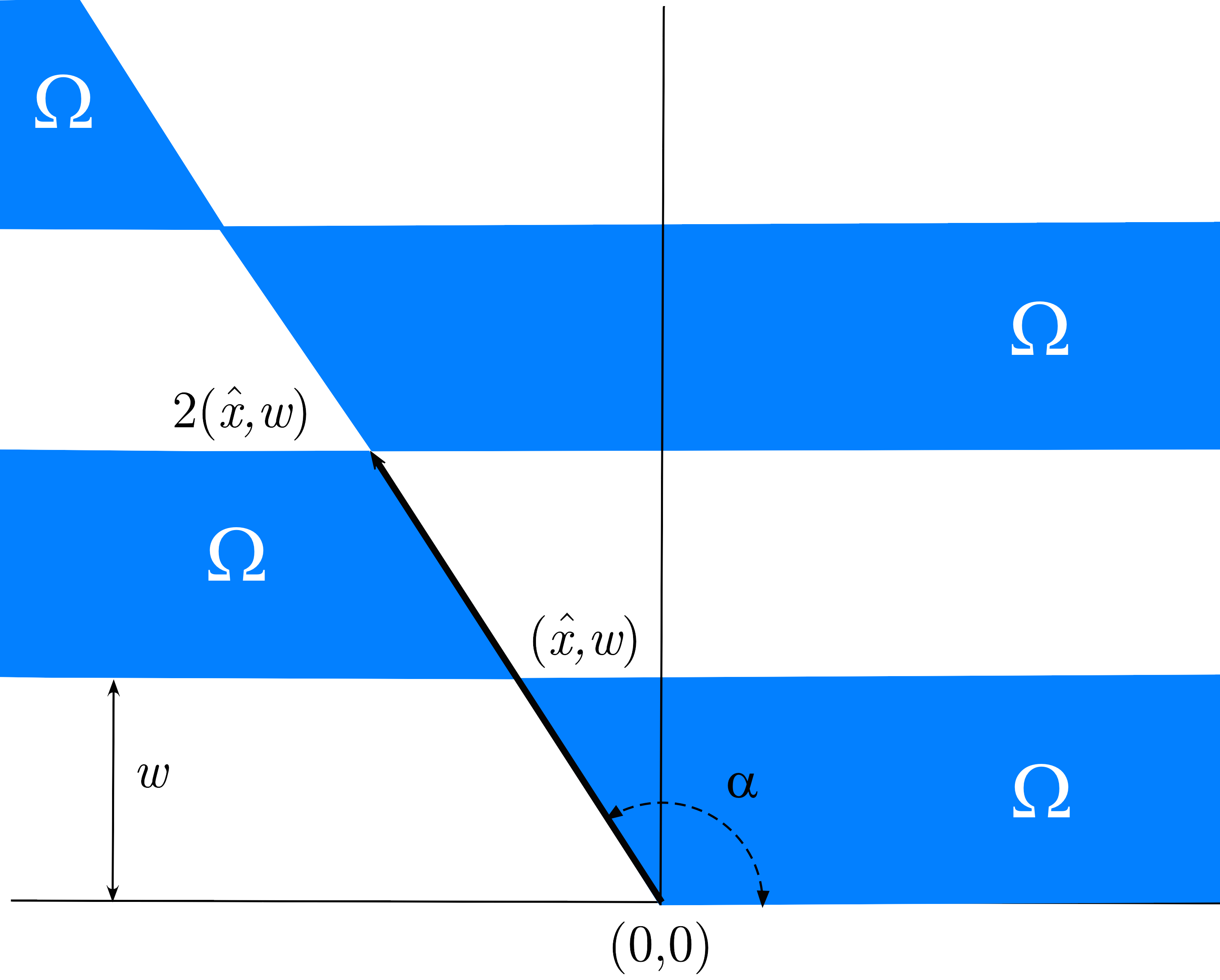}  
\caption{\small The domain $\Omega$ of a scherkenoid.}
\label{fig-new}
\end{center}
\end{figure} 

\begin{proof}
Let $L_n$ be the vertical line through the point $(n\hat{x},nw)$.
Then the surface $\Ss_{\alpha,w} \setminus \cup_nL_n$ is the graph of a function 
\[
    u: \Omega\to \RR
\]
where $\Omega$ is the domain shown in Figure~\ref{fig-new}.
The function $u$ is periodic with period $(2\hat{x},2w)$.
By Corollary~\ref{scherkenoid-corollary},
\[
    \pdf{u}{x} < -\sqrt{(w/\pi)^2 - 1} \quad \text{on $[0,\infty)\times (0,\pi)$}.
\]
Any point in $\Omega\cap ([c,\infty)\times [w,\infty))$ can be written uniquely as $(x,y+ 2wn)$
for some $(x,y)\in [c, \infty)\times (0,w)$ and some positive integer $n$.  By the periodicity,
\begin{align*}
u(x,y+2wn)
&=
u(x - 2\hat{x}n, y) 
\\
&\le
u(1,y)  -  (x - 2\hat{x}n -1) \sqrt{(w/\pi)^2 - 1} 
\\
&\le
u(1,y) - (c - 2\hat{x} - 1)  \sqrt{(w/\pi)^2 - 1}.
\end{align*}
\end{proof}



 \section{Translating helicoids}
 \label{helicoid-section}

\begin{theorem}\label{helicoid-theorem}
Suppose $M$ is a smooth translator such that $\partial M$ consists of two vertical lines,
and such that $M\setminus \partial M$ is the graph of a function
\[
   u: \RR\times (0,w)\to \RR
\]
with boundary values
\[
   u(x,0) =
   \begin{cases} 
   +\infty &\text{if $x<0$} \\
   -\infty &\text{if $x>0$}
   \end{cases}
\quad
,
\quad
   u(x,w) =
   \begin{cases} 
   -\infty &\text{if $x<\hat{x}$} \\
   +\infty &\text{if $x>\hat{x}$}
   \end{cases}
\]
for some $\hat{x}\in\RR$.
Let $\widehat{M}$ be the surface obtained from $M$ by repeated Schwarz reflection,
and let $\Ll$ be the line through $(0,0,0)$ and $(\hat{x},w,0)$.
Then
\begin{enumerate}[\upshape (1)]
\item\label{helicoid-flat-item}
If $p_i= (x_i,y_i,z_i)\in M$ and if $|x_i|\to -\infty$,
then $M-p_i$ converges smoothly to the plane $y=0$.
The unit normal $\nu(p_i)$ converges to $\ee_2$ if $x_i\to -\infty$ and to $-\ee_2$ if $x_i\to \infty$.
\item\label{helicoid-curvature-item}
    $M$ has negative Gauss curvature at every point.
\item\label{helicoid-gauss-image-item}
 The Gauss map is a diffeomorphism from $M$ onto $\SS^{2+}\setminus\{\ee_2,-\ee_2\}$,
where $\SS^{2+}$ is the closed upper hemisphere.
\item\label{helicoid-periodic-item}
 The surface $\widehat{M}$ is singly periodic with period $2(\hat{x},w,0)$.  
If $p_i \in \widehat{M}$ and if $\dist(p_i,\Ll)\to\infty$, then 
then $M-p_i$ converges smoothly to the parallel planes $y=nw$, $n\in \ZZ$.
\item\label{helicoid-w-item}
            $w<\pi$.
\item\label{helicoid-hat-x-item} $\hat{x}>0$.
\end{enumerate}
\end{theorem}
 \begin{figure}[htbp]
\begin{center}
\includegraphics[height=.36\textheight]{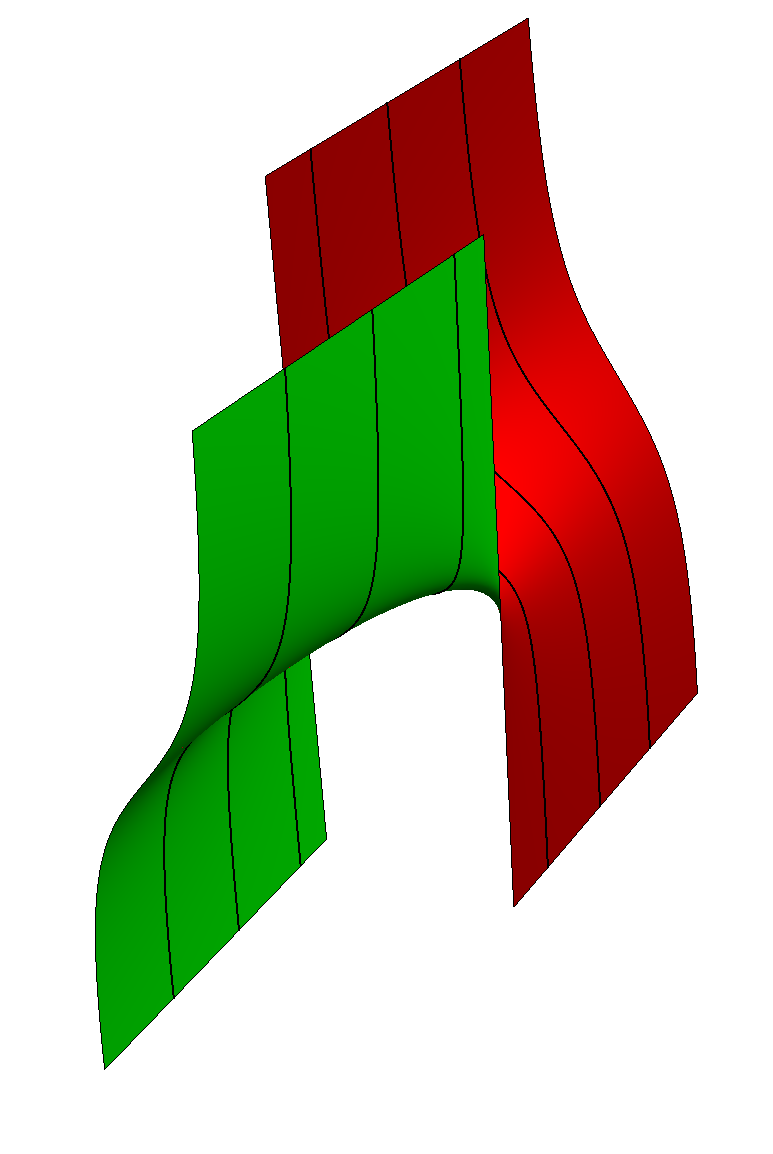}
\caption{The graph of a function $u$ for a helicoid-like translator in Theorem \ref{helicoid-theorem} 
with  $w= \pi/2$.}
\label{default-3}
\end{center}
\end{figure}
\begin{remark}
By Theorem~\ref{pi-theorem}, there exists such an example for each $w$ with $0<w<\pi$;
indeed there exists an example with the additional property that $M$
is invariant under rotation by $\pi$ about the vertical line through $(\hat{x}/2,w/2)$.
We do not know any results about uniqueness.
\end{remark}

We say that the line $\Ll$ is the {\bf axis} of the helicoid-like surface $\widehat{M}$.
By contrast with the ordinary helicoid, the axis is not perpendicular to the planes~$y=nw$;
see Proposition~\ref{axis-proposition} below.
Also, it seems unlikely that $\widehat{M}$ contains the axis or any other non-vertical line.

\begin{proof}[Proof of Theorem~\ref{helicoid-theorem}]
Assertion~\ref{helicoid-flat-item} is a special case of Theorem~\ref{vrock}.
Assertion~\ref{helicoid-curvature-item} (negative curvature everywhere) is 
a special case of Theorem~\ref{other-negative-theorem}.

If $p_i$ is a divergent sequence in $M\setminus \partial M$, then (after passing to a subsequence)
$\nu(p_i)$ converges to a limit $\vv\in \SS^2$.  After passing to a further subsequence,
either $p_i$ diverges in $M$ or $p_i$ converges to a point in $\partial M$.
In the first case, $\vv$ is $\ee_2$ or $-\ee_2$ by Assertion~\ref{helicoid-flat-item} and by the boundary
values of $u$. In the second case, $\vv$ is also in the equator $\partial \SS^{2+}$ (since the components
of $\partial M$ are vertical lines).
Thus 
\[
    \nu:  M\setminus\partial M  \to \SS^{2+}\setminus \partial \SS^{2+} 
    \tag{*}
\]
is a proper immersion, i.e., a covering map.  Since the domain is connected and the target is 
simply connected, the map~\thetag{*} is a diffeomorphism.
Assertion~\ref{helicoid-gauss-image-item} follows, since $\nu$ maps $\partial M$
diffeomorphically onto $(\partial \SS^{2+})\setminus \{\ee_2,-\ee_2\}$.

Assertion~\ref{helicoid-periodic-item} follows easily from the preceding assertions.

We now prove Assertion~\ref{helicoid-w-item}: $w<\pi$.
By Assertion~\ref{helicoid-gauss-image-item}, there is a point $p_0=(x_0,y_0)$
such that $Du(x_0,y_0)=0$.  Thus there is a grim reaper surface over the strip 
\[
    S = \RR\times (y_0-\tfrac{\pi}2, y_0+ \tfrac{\pi}2)
\]
that is tangent to $M$.  By Theorem~\ref{other-tangency-theorem}, the strip $S$
must intersect both of the boundary rays on which $u= -\infty$.  Thus
\[
    y_0-\tfrac{\pi}2 < 0  \quad\text{and}\quad w < y_0+\tfrac{\pi}2,
\]
which implies that $w< \pi$.

Assertion~\ref{helicoid-hat-x-item} ($\hat{x}>0$) follows from Proposition~\ref{axis-proposition} below.
\end{proof}

\begin{proposition}\label{axis-proposition}
Let $u$ and $\hat{x}$ be as in Theorem~\ref{helicoid-theorem}.  
Then
\[
   \hat{x} = \frac12\iint_{\RR\times (0,\pi)} \frac1{\sqrt{1+|\nabla u|^2}}\,dx\,dy.
\]
In particular, $\hat{x}>0$.
\end{proposition}

\begin{proof}
Note that $u$ solves the translator equation
\[
   \Div \left( \frac{\nabla u}{\sqrt{1+ |\nabla u|^2}} \right) = - \frac{1}{\sqrt{1+ |\nabla u|^2}}.
\]
For $a>|\hat{x}|$, integrating over the rectangle $\Rr(a)=[-a,a]\times [0,w]$  
and using the divergence theorem gives
\begin{align*}
   &-2\hat{x}
   +
   \int_0^w \frac{\partial u/ \partial x}{\sqrt{1+|Du|^2}}(a,y)\,dy
   -
   \int_0^w \frac{\partial u/ \partial x}{\sqrt{1+|Du|^2}}(-a,y)\,dy
\\
   =
   &-\iint_{\Rr(a)} \frac1{\sqrt{1+ |Du|^2}}\,dx\,dy.
\end{align*}
In the notation of Theorem~\ref{helicoid-theorem}, the integrands on
the left are equal to
\[
    - \nu(a,y)\cdot \ee_1 \quad\text{and}\quad  \nu(-a,y)\cdot \ee_1
\]
which tend to $0$ uniformly
as $a\to\infty$ according to that theorem.
\end{proof}


\section{Pitchforks}\label{pitchfork-section}

\begin{theorem}\label{pitchfork-theorem}
Let $M$ be a smooth translator such that $\partial M$ is the $z$-axis $Z$ 
and such that $M\setminus \partial M$
is the graph of a function
\[
    u: \RR\times (0,w)\to \RR
\]
with boundary values
\[
   u(x,0) =
   \begin{cases} 
   +\infty &\text{if $x<0$} \\
   -\infty &\text{if $x>0$}
   \end{cases}
\quad
,
\quad
   u(\cdot,w) = -\infty.
\]
Then
\begin{enumerate}[\upshape (1)]
\item\label{pitchfork-width-item} $w\ge \pi$.
\item\label{pitchfork-curvature-item}
 $M$ has negative Gauss curvature everywhere.
\item\label{pitchfork-flat-item}
If $p_i= (x_i,y_i,z_i)\in M$ and if $x_i \to -\infty$,
then $M-p_i$ converges smoothly to the plane $y=0$.
\item\label{pitchfork-reaper-item}
As $a\to\infty$, 
\[
    u(a+x,y) - u(a,y)
\]
converges smoothly to the unique tilted grim reaper
$
  g_w: \RR\times (0,w)\to \RR
$
such that $g_w(0,w/2)=0$ and $\partial g_w/\partial x \le 0$.
\item\label{pitchfork-gauss-image-item}
The Gauss map gives a diffeomorphism from $M\setminus \partial M$ onto 
the region $\Rr=\Rr(w)$ in the upper hemisphere bounded by $C\cup C(w)$,
where $C$ is the equatorial semicircle
\[
  C = \{(x,y,0)\in \SS^2: x\ge 0\}
\]
and where $C(w)$ is the great semicircle that is the image of the graph of $g_w$ under its Gauss map:
\[
    C(w) =  \left\{ (x,y,z)\in \SS^2: \text{$z>0$ and $x = z\,\sqrt{(w/\pi)^2-1}$} \right\}.
\]
\end{enumerate}
\end{theorem}
\begin{figure}[htbp]
\begin{center}
\includegraphics[width=.46\textwidth]{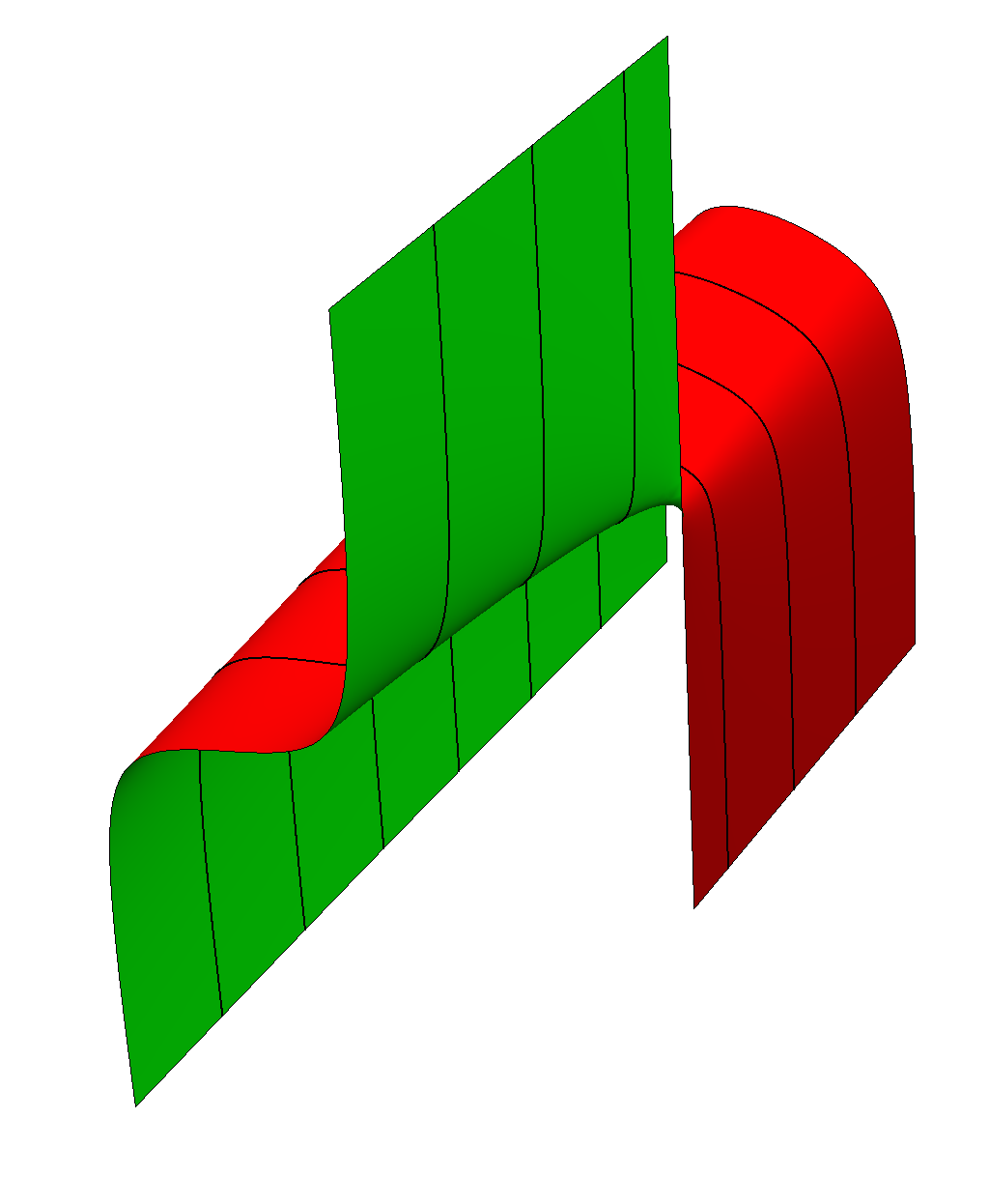}
\caption{The graph of a pitchfork function $u$ in Theorem \ref{pitchfork-theorem} with  $w= \pi$.}
\label{default-2}
\end{center}
\end{figure}

\begin{corollary}
The Gauss map is a diffeomorphism from $M$ onto
\[
     \Rr \cup (C\setminus \{\ee_2,-\ee_2\}),
\]
and the integral of the absolute Gauss curvature of $M$ is $2\arcsin(\pi/w)$.
If $\widehat{M}$ is the complete surface obtained from $M$ by Schwarz reflection
about the $z$-axis, then the Gauss map gives a diffeomorphism from $\widehat{M}$
onto 
\[
     \Rr \cup (C\setminus \{\ee_2,-\ee_2\}) \cup \widetilde{\Rr},
\]
where $\widetilde{\Rr}$ is the image of $\Rr$ under $(x,y,z)\mapsto (x,y,-z)$.
\end{corollary}

\begin{remark} By Theorem~\ref{pi-theorem},
 there exists at least one such an example $M$ for every $w\ge \pi$.
We do not know whether $M$ is unique.
\end{remark}

\begin{proof}
Assertion~\ref{pitchfork-width-item} ($w\ge \pi$) follows from Theorem~\ref{half-reaper}.
Assertion~\ref{pitchfork-curvature-item} (negative curvature everywhere) is 
a special case of Theorem~\ref{other-negative-theorem}.
Assertion~\ref{pitchfork-flat-item} follows from Theorem~\ref{vrock}.

Concerning Assertion~\ref{pitchfork-reaper-item}, we see from Theorem~\ref{half-reaper} that
\[
   u(a+x,y) - u(a,w/2)
\]
converges smoothly as $a\to\infty$ either to $g_w$ or to the corresponding
upward-sloping grim reaper
\[
   g_w': \RR\times (0,w)\to \RR
\]
with $\partial g_w'/\partial x \ge 0$.  If $w=\pi$, then $g_w=g_w'$.
Thus to prove Assertion~\ref{pitchfork-reaper-item},
it suffices to prove for $w>\pi$ that there is a $\hat{y}\in (0,w)$ such that
\[
  \lim_{x\to\infty} \pdf{u}x(x,\hat{y}) \le 0.
\]
This is done in Lemma~\ref{pitchfork-lemma} below.

The proof of Assertion~\ref{pitchfork-gauss-image-item} about the Gauss map image 
is identical to the proof of the corresponding assertion for the Scherkenoid
(Theorem~\ref{scherkenoid-theorem}.)
\end{proof}


\begin{lemma}\label{pitchfork-lemma}
Let $L$ be a straight line that intersects the negative $x$-axis and whose distance from the origin is $> \pi$.
Let $\vv=(v_1,v_2)$ be a vector tangent to $L$ with $v_2>0$.  
Then 
\begin{equation}\label{tilted-line}
    \nabla u\cdot \vv<0 \quad\text{at all points $(x,y)\in L$ where $0<y<w$,}
\end{equation}
and
\begin{equation}\label{untilted}
\pdf{u}x \le 0 \quad\text{at all points in $\RR\times (\pi,w)$}.
\end{equation}
\end{lemma}

\begin{proof}
Suppose, contrary to~\eqref{tilted-line}, that there is a point on $L$ at which $\nabla u\cdot \vv\ge 0$.
Since $u=+\infty$ on the point where $L$ intersects $\{y=0\}$, there is a point $p_0$ on $L$
such that $\nabla u\cdot \vv = 0$. 
Thus there is a grim reaper surface $G$ over an open strip $S$ containing $L$ such that $G$
and $M$ at tangent at $p_0$.

Since $S$ has width $\pi$ and since $\dist(L,(0,0))> \pi$, we see that $\dist(S,(0,0))>0$.
In particular, $S$ is disjoint from the ray $[0,\infty)\times \{0\}$, the portion of the $x$-axis where $u=-\infty$.
But then $M$ and $G$ cannot be tangent by Theorem~\ref{other-tangency-theorem},
a contradiction. 
 This proves~\eqref{tilted-line}.

Now let $\pi<b<w$.  For all sufficiently small $m>0$, 
the line $L$ given by $y=mx+b$ has distance $>\pi$ from $(0,0)$.
Thus by~\eqref{tilted-line},
\[
    \left( \frac{d}{dx} \right) u(x,mx+b) < 0  \quad\text{for all $-b/m < x < (w-b)/m$}.
\]
Letting $m\to 0$ gives $(\partial/\partial x)u(x,b)\le 0$ for all $x$.  Since this holds for all $b\in (\pi,w)$, we have
proved~\eqref{untilted}.
\end{proof}



\appendix

\section{Curvature Bounds, Existence, and Uniqueness}

In this appendix, we prove the curvature and area estimates for translating graphs  
that are used in the paper, and we prove existence and uniqueness of translating graphs
over polygonal domains with boundary values that are constant on each edge.
The proofs are rather standard, but we have included them for completeness.

\begin{theorem}\label{minimizing-theorem}
Let $M$ be a translator that is the graph of a function $u:\Omega\to \RR$, where $\Omega$ is a convex
open set in $\RR^2$. If $\Sigma$ is a compact region in $M$ and if $\Sigma'$ is a compact surface
with $\partial \Sigma'=\partial \Sigma$, then 
\[
   \area_g\Sigma \le \area_g\Sigma'. \tag{*}
\]
\end{theorem}

\begin{proof}
Since $M$ and its vertical translates form a $g$-minimal foliation of $\Omega\times\RR$,
the theorem is true if $\Sigma'$ is contained in $\Omega\times\RR$ 
\cite{solomon}*{Corollary~1.11}.

For the general case, 
let $C$ be the convex hull of the projection of $ \Sigma$ to the $xy$-plane. 
Let $\Pi:\RR^3\to C\times\RR$ map each point $p$ to the point  $q\in C\times\RR$ 
that minimizes $|q-p|$.  Then $\Pi$ is distance decreasing (and therefore area decreasing)
 with respect to the $g$-metric, so
 \[
    \area_g(\Pi(\Sigma')) \le \area_g(\Sigma').
 \]
Thus the general case reduces to the special case of surfaces in $\Omega\times\RR$.
\end{proof}

\begin{corollary}\label{minimizing-corollary}
If $U$ is a bounded convex open subset of $\RR^3$ disjoint from $\Gamma:=\overline{M}\setminus M$, then
\[
  \area_g(M\cap U) \le \frac12 \area_g(\partial U).
\]
\end{corollary}

\begin{proof}
By approximation, it suffices to prove it when $\overline{U}$ is disjoint from $\Gamma$.
Since $M\cap \overline{U}$
is $g$-area-minimizing, it has less $g$ area than either of the two regions in $(\partial U)\setminus M$
with the same boundary as $M\cap \overline{U}$.  In particular, the $g$-area of $M$ is less than the average of the areas
of those two regions.
\end{proof}

\begin{theorem}\label{curvature-bound}
There is a constant $C<\infty$ with the following property.
Let $M$ be translator with velocity $-s\ee_3$ in $\RR^3$ (where $s>0$) such that
\begin{enumerate}
\item $M$ is the graph of a smooth function $F:\Omega\to\RR$ on a convex open subset $\Omega$ of $\RR^2$.
\item $\Gamma:=\overline{M}\setminus M$ is a polygonal curve (not necessarily connected)
       consisting of segments, rays, and lines.
\item $\overline{M}$ is a smooth manifold-with-boundary except at the corners of $\Gamma$.
\end{enumerate}
If $p\in \RR^3$, let $r(M,p)$ be the supremum of $r>0$ such that $\BB(p,r)\cap \partial M$ is either empty
or consists of a single line segment.
Then
\[
   A(M,p) \min \{s^{-1}, r(M,p) \} \le C,
\]
where $A(M,p)$ is the norm of the second fundamental form of $M$ at $p$.
\end{theorem}

\begin{proof}
Suppose the theorem is false.  Then there is a sequence of examples $p_i\in M_i$ 
with
\begin{equation}\label{blows-up}
  A(M_i,p_i)\, \min\{ s_i^{-1}, r(M_i,p_i) \} \to \infty.
\end{equation}
Let $\BB_i=\BB(p_i,R_i)$ where
\begin{equation}\label{R-equals}
  R_i = \frac12 \min\{s_i^{-1}, r(M_i,p_i)\}.
\end{equation}
Let $q_i$ be a point in $\overline{M}\cap \BB_i$ that 
maximizes\footnote{The factor of $1/2$ in~\eqref{R-equals} is to guarantee that the maximum will be attained.
Any positive factor $<1$ would work.}
\[
   A(M_i,q_i) \, \dist(q_i, \partial \BB_i).
\]
By translating and scaling, we can assume that $q_i=0$ and that
$A(M_i,0)=1$. 
Thus 
\begin{align*}
  \dist(0,\partial \BB_i) 
  &= A(M_i,0) \, \dist(0,\partial \BB_i) \\
  &\ge  A(M_i,p_i) \,\dist(p_i,\partial \BB_i) \\
  &= A(M_i,p_i)  \, R_i  \\
  &\to \infty
\end{align*}
by~\eqref{blows-up} and~\eqref{R-equals}.

Since $R_i\ge \dist(0,\partial\BB_i)$, we see that $R_i$, $s_i^{-1}$, and $r(M_i,0)$ all tend to infinity.

Note that $\BB_i\cap \partial M$ is either a line segment or the empty set.
Note also that $A(M_i,\cdot)$ is bounded as $i\to\infty$ on compact subsets of $\RR^3$.
Thus, by passing to a subsequence, we can assume that the $M_i$ converge smoothly to limit $M$
with $A(M,0)=1$, where $\partial M$ is either empty or a straight line.

Since $s_i\to 0$, $M$ is area-minimizing and therefore stable with respect to the Euclidean metric.
Since complete stable minimal surfaces in $\R^3$ are flat and $A(M,0)=1$, we see that $M$
must have boundary.  Thus $L=\partial M$ is a straight line.  The convexity of $\Omega$
implies that $M$ is contained in a closed halfspace $\Hh$ bounded by a vertical plane $\partial \Hh$ containing $L$.

By Corollary~\ref{minimizing-corollary},
 the area of $M\cap (\BB(0,R)\cap \Hh)$ is at most half the area of the boundary of $\BB(0,R)\cap \Hh$.
Thus $M$ has quadratic area growth.  Also, $M$ is simply connected since the $M_i$ are simply connected.
Thus $M\cup \rho_LM$ is a complete, properly embedded, simply connected minimal surface with quadratic
area growth.  The only such minimal surface is the plane, contradicting $A(M,0)=1$.
\end{proof}

\begin{remark}\label{dichotomy}
Let $M_i$, $\Gamma_i=\overline{M_i}\setminus M_i$, and $\Omega_i$ be a sequence
of examples satisfying the hypotheses of Theorem~\ref{curvature-bound}
 with $s_i\equiv 1$.  Suppose that the $\Gamma_i$
converge (with multiplicity $1$) to a polygonal curve $\Gamma$.  Thus curvature estimates imply that
(after passing to a subsequence) the $M_i$ converge smoothly in $\RR^3\setminus\Gamma$
to a smooth translator $M$.  By Corollary~\ref{minimizing-corollary}, $M$ is embedded with multiplicity $1$.
Let $M_c$ be a connected component of $M$.  
Note that vertical translation gives a $g$-Jacobi field on $M$ that does not change sign (since $M$ is a limit
of graphs.)  
 By the strong maximum principle, if it vanishes anywhere on $M_c$, it would vanish everwhere on $M_c$.
 In that case, the translator equation implies that $M_c$ is flat.
Thus each connected component $M_c$ of $M$ is either a graph or is flat and vertical.
\end{remark}

\begin{theorem}\label{finite-existence}
Let $\Omega$ be the interior of a convex polygon in $\RR^2$.  Let $V$ be the set of vertices of the polygon.
Let $f:(\partial \Omega)\setminus V\to \RR$ be a function that is a (finite) 
constant on each edge, and let $\Gamma$
be the simple closed curve obtained by adding vertical segments (as needed) to the graph of $f$.
Then there is a unique embedded $g$-minimal disk $M$ with $\partial M=\Gamma$
such that $M$ is a smooth manifold-with-boundary except at the corners of $\Gamma$
and such that $M\setminus \Gamma$ is the graph of a smooth function $F: \Omega\to \RR$.
\end{theorem}

\begin{proof}
First we prove existence.
Let $K$ be the closed region under a bowl soliton and above a horizontal plane, where the soliton and the plane
are chosen so that $K$ contains $\Gamma$.

Then $N:=K\cap (\overline{\Omega}\times\RR)$ is compact and $g$-mean convex.

Thus by~\cite{meeks-yau}, 
there is a least $g$-area disk $M$ in $N$ with boundary $\Gamma$,
and it is a smooth embedded manifold-with-boundary except at the corners of $\Gamma$.

To show that $M\setminus \partial M$ is a graph, it suffices to show that the tangent plane
is never vertical.  Suppose to the contrary that $\nu=(a,b,0)$ is a nonzero normal to $M$ at
a point $p\in M\setminus\partial M$.   Then 
\[
    h: (x,y,z)\in M\mapsto ax+by
\]
is a nonconstant harmonic function on the disk $M$ with a critical point at $p$, so
by a basic property of harmonic functions (Rado),
\[
   \Gamma \cap \{ h> h(p)\} \tag{*}
\]
has two or more connected components.   But by the convexity of the polygon $\Omega$,
\thetag{*} is connected.  The contradiction proves that $M$ is a graph.

Now we prove uniqueness.  Consider a corner $p$ of $\Gamma$ at which a horizontal edge
meets a vertical edge.  Note that $p$ lies above a corner of the polygon $\partial \Omega$.
Since $M\subset \overline{\Omega}\times\RR$, we see that the tangent cone to $M$ at $p$
must be a quarter of plane (and not, for example, three quarters of a plane.)
Thus Schwarz reflection around the vertical edge produces a surface $M^*$ whose tangent
cone at $p$ is a halfplane.  Thus $M^*$ is smooth at $p$.

In particular, at the point $p$, $M$ is tangent to a face of $\Omega\times \RR$.

Now suppose that uniqueness fails, i.e., that there are two distinct such surfaces $M$ and $M'$.
We may suppose that $M$ contains a point that lies above $M'$. 

Choose $z$ large enough that $M'(z):=M'+(0,0,z)$ is disjoint from $M$.  Now decrease
$z$ until the first time that $M'(z)$ touches $M$ at some interior point or that $M'(z)$ and $M$ are tangent at some boundary point.

Note that $z>0$ since $M$ contains a point that lies above $M'=M'(0)$.

By the strong maximum principle, $M'(z)$ cannot touch $M$ at an interior point.

Thus $M'(z)$ and $M$ must be tangent at a boundary point $p$. Note that $p$ must lie on a vertical segment $S$
of $\Gamma$.

There are four cases:
\begin{enumerate}[\upshape (1)]
\item $p$ is not a corner of $\partial M'(z)$ or of $\partial M$.
\item $p$ is a corner $\partial M'(z)$ but  not of $\partial M$.
\item $p$ is a corner of $\partial M(z)$ but not of $\partial M'(z)$.
\item $p$ is a corner of $\partial M'(z)$ and a corner of $\partial M$.
\end{enumerate}
Case 1 is impossible by the strong boundary maximum principle. Case 2 is impossible since at a corner of $\partial M'(z)$, $M'(z)$ is tangent to a face of $\Omega\times\RR$,
whereas $M$ cannot be tangent to a face at a noncorner by the boundary maximum principle.
Case 3 is just like Case 2. 
Case 4 is impossible since, if it occurred, $M'(z)$ would be tangent to one face of $\Omega\times\RR$
at $p$ whereas $M$ would be tangent to the adjacent face, and thus $M'(z)$ and $M$ would not be
tangent to each other at $p$.
\end{proof}

\end{document}